\documentclass[11pt]{article}

\usepackage{amsmath,amsfonts,amsthm,amsrefs}

\usepackage{graphicx,color,pict2e}

\definecolor{refkey}{rgb}{1, 0.5, 0}  
\definecolor{labelkey}{rgb}{1,1,1}
\definecolor{labelkey}{rgb}{0.1,1,0.2}

\newtheorem{theorem}{Theorem}[section]
\newtheorem{lemma}{Lemma}[section]
\newtheorem{remark}{Remark}[section]
\newtheorem{corollary}{Corollary}[section]
\newtheorem{definition}{Definition}[section]

\def\bel{\begin{equation}\label}
\def\eeq{\end{equation}}
\def\bega{\begin{array}}
\def\enda{\end{array}}

\def\ve{\varepsilon}

\title{Traveling Wave Profiles for a Follow-the-Leader Model for Traffic Flow  with Rough Road Condition}

\author{Wen Shen\footnote{Mathematics Department, Pennsylvania State University, USA. 
wxs27@psu.edu}}

\begin{document}

\maketitle

\begin{abstract}
We study a Follow-the-Leader (FtL) ODE model for traffic flow with rough road
condition, and analyze stationary traveling wave profiles where the solutions of the FtL model
trace along, near the jump in the road condition.  
We derive a discontinuous delay differential equation (DDDE) for these profiles.
For various cases, we obtain results on existence, uniqueness and local stability 
of the profiles. 
The results here offer an alternative approximation, 
possibly more realistic than the classical vanishing viscosity approach, 
to the conservation law with discontinuous flux for traffic flow. 
\end{abstract}


%%%%%%%%%%%%%%%%%%%%%%%%%%%%
\section{Introduction}\setcounter{equation}{0}
%%%%%%%%%%%%%%%%%%%%%%%%%%

We consider an ODE model for traffic flow with rough road condition.
Given an index $i\in\mathbb{Z}$ and a time $t\ge0$, 
let $z_i(t)$ be the position of car number $i$ at time $t$. 
Let $\ell$ be the length of all cars, so that 
\[ z_i (t)+\ell  \le z_{i+1}(t) , \qquad \forall t,i,\]
one defines a discrete local density $\rho_i(t)$ for each car with index $i$:
\bel{def-rho-i}\rho_i (t) \;\dot=\; \frac{\ell}{z_{i+1}(t)-z_i(t)}.
\eeq
By this normalized 
definition, the maximum car density is $\rho=1$ where cars are bumper-to-bumper.

The road condition includes many factors,  for example  the number of lanes,
quality of the road surface, surrounding situation, among other things. 
For simplicity of the discussion, we
let $k(x)$ be the speed limit which reflects the various road conditions.
We are particularly interested in the case where $k(x)$ is discontinuous. 

At time $t$, 
given a distribution of car positions $\{z_i(t)\}$, the speed of each car
is determined by its discrete local density and the road condition:
\bel{dotzi}
\dot z_i (t) = k(z_i(t)) \cdot \phi(\rho_i(t)).
\eeq
Here $\phi(\rho)$ is a decreasing function, with 
\bel{phi-prop}
 \phi' (\rho)\le -\hat c_0 <0, \qquad \phi(1)=0, \quad \phi(0)=1. 
 \eeq
For example, the popular 
Lighthill-Whitham model~\cite{MR0072606} uses, 
\begin{equation}\label{LW}
\phi(\rho)=1-\rho.
\end{equation}

The system of ODEs~\eqref{dotzi} describes the {\em Follow-the-Leader} behavior, 
and is referred to as the \textbf{FtL} model.
By simple computation we obtain
an equivalent system of ODEs for the local densities $\rho_i$:
\bel{dot-rho-i}
\dot\rho_i = \frac{\ell}{(z_{i+1}-z_i)^2} \Big[  \dot z_i - \dot z_{i+1}\Big]
= \frac{\rho_i^2}{\ell} \; 
\Big[ k(z_{i}) \phi(\rho_{i}) - k(z_{i+1}) \phi(\rho_{i+1}) \Big] .
\eeq
Note that given the set $\{\rho_i\}$, one can recover the set for the car positions 
$\{z_i\}$ by $z_{i+1}=z_i + \ell/\rho_i$.
The car position distribution $\{z_i\}$ is unique if we fix any car, 
say $z_0=0$. 

\bigskip

Let $\{z_i(t),\rho_i(t)\}$  denote the solution of the FtL model. 
We seek stationary profiles $Q(x)$ such that 
the points $\{z_i(t),\rho_i(t)\}$ trace along the graph of $Q(x)$. 
To be specific, we require
\bel{Q1}
 Q(z_i(t)) = \rho_i(t) \qquad \forall i,t.
 \eeq
Differentiating~\eqref{Q1} in $t$, and using~\eqref{dotzi} and~\eqref{dot-rho-i},
we obtain
\[
Q'(z_i) = \frac{\dot \rho_i}{\dot z_i} 
= \frac{\rho_i^2}{\ell\cdot k(z_i) \, \phi(\rho_i)} 
\Big[ k(z_i)\phi(\rho_i) - k(z_{i+1}) \phi(\rho_{i+1})\Big].
\]
Using
\[ 
z_{i+1}=z_i+\frac{\ell}{\rho_i}, \qquad \rho_{i}=Q(z_{i}),
\]
and writing $x$ for $z_i$ (since it is arbitrary), we get
\bel{DQ}
Q'(x) = \frac{Q(x)^2}{\ell \; k(x)\phi(Q(x))} \cdot \left[k(x)\phi(Q(x)) - 
k(x^\sharp) \phi(Q(x^\sharp) )\right],
\qquad 
x^\sharp=x+\frac{\ell}{Q(x)}.
\eeq
Here $x^\sharp$ is the location of the ``leader'' for the car located at $x$. 
In the literature,~\eqref{DQ} belongs to a type of equations which is
called a {\em delay differential equation} (DDE), or
a {\em differential equation with retarded argument}.

When the road condition is uniform so that $k(x)\equiv V$ is constant, it is known that 
the solutions of the FtL model~\eqref{dot-rho-i} 
converge to the scalar conservation law 
(cf.~\cites{MR3356989,MR3217759,HoldenRisebro,HoldenRisebro2} and references there in)
\begin{equation}\label{claw-U}
\rho_t + f(\rho)_x =0, \qquad f(\rho)\;\dot=\; V\rho \cdot \phi(\rho), 
\end{equation}
as $\ell\to 0+$,  under suitable assumptions on the initial data. 
In the literature it is customary  to consider the flux $f$ a concave function
with
\bel{fdd} f'' \le -c_0<0.
\eeq
This leads to the following reasonable assumption on $\phi$:
\bel{phi-prop2}
-\phi''(\rho) > \frac{1}{\rho} \left[ 2\phi'(\rho)+c_0/V \right].
\eeq

In this simpler case where $k(x)=V$, equation~\eqref{DQ} takes a simpler form.
Let $W(x)$ denote this stationary profile.  We have
\bel{DW}
W'(x) = \frac{W(x)^2}{\ell \cdot \phi(W(x))} \cdot 
\left[ \phi(W(x)) - \phi(W(x^\sharp))\right], \qquad x^\sharp=x+\frac{\ell}{W(x)}.
\eeq
Equation~\eqref{DW} is studied  by the author and collaborator 
in~\cite{ShenKarim2017}, where
we establish the existence and uniqueness (up to horizontal shifte) 
of the
profile $W(x)$, connecting two ``boundary'' conditions at the infinities
\[
 \lim_{x\to\pm\infty} W(x) =  \rho_\pm, \quad \mbox{where}
\quad
 0\le \rho_- \le \rho^*\le \rho_+\le 1, \quad f(\rho_-)=f(\rho_+), \quad 
 f'(\rho^*)=0.
\]
We show that 
the profile $W(x)$ is monotone and 
approaches $\rho_\pm$ at an exponential rate.
%We also show that 
%if $\rho_- >\rho^*$ or $\rho_+<\rho^*$, 
%these asymptotes are unstable.
Furthermore, we  prove that the profile $W(x)$ is a local attractor
for nearby solutions of the FtL model.

\medskip

In this paper we consider rough road condition, and analyze the behavior of solutions
in the neighborhood of a discontinuity in $k(x)$. 
To fix the idea, we consider the case where $k(x)$ is piecewise constant and 
has a jump at $x=0$, i.e.,
\begin{equation}\label{eq:V0}
k(x) = \begin{cases} V_+,  \quad& (x \ge 0), \\ 
V_-, & (x<0) . \end{cases}
\end{equation}
The ODEs for $\rho_i$ in~\eqref{dot-rho-i} take the following form
\begin{equation}\label{dot-rho-2}
\dot \rho_i ~=~ 
\begin{cases}
\displaystyle \ell^{-1} V_-\rho_i^2\; \Big[\phi(\rho_{i})-\phi(\rho_{i+1})\Big], 
& ( z_i < z_{i+1}<0),\\[2mm]
\displaystyle \ell^{-1} \rho_i^2 \;\Big[ V_- \phi(\rho_i)-V_+\phi(\rho_{i+1})\Big],
 \qquad & (z_i <0 \le z_{i+1}),\\[2mm]
\displaystyle \ell^{-1} V_+ \rho_i^2\; \Big[\phi(\rho_{i})-\phi(\rho_{i+1})\Big], 
\quad & (0 \le z_i < z_{i+1}).
\end{cases}
\end{equation}
The system of ODEs in~\eqref{dot-rho-2} has discontinuous right hand side. 
The discontinuity occurs twice for each $\rho_i$, 
as the car position $z_i$ crosses $x=0$,  and as its leader $z_{i+1}$ crosses $x=0$.

The corresponding profile $Q(x)$ satisfies the following 
{\em discontinuous delay differential equation} (DDDE):
\begin{equation}\label{DQ-2}
Q'(x) 
= 
\begin{cases}
\displaystyle \frac{Q(x)^2}{\ell \phi(Q(x))} \left[ \phi(Q(x))-\phi(Q(x^\sharp))\right] ,
&\displaystyle
 (x^\sharp <0 \quad \mbox{or}\quad x>0) , \\[3mm]
\displaystyle \frac{Q(x)^2}{\ell V_-\phi(Q(x))} \left[ V_-\phi(Q(x))-V_+\phi(Q(x^\sharp))\right], 
\quad   &\displaystyle
 (x<0<x^\sharp) , 
 \end{cases}
\end{equation} 
where 
\[x^\sharp=x+\ell/Q(x)\]
is the position for the leader of the car at $x$. 
Note that for the first  case in~\eqref{DQ-2} the equation is the same as~\eqref{DW}. 
For the second case, where the car is behind the jump in $k(x)$ but the leader 
is ahead of the jump,  the equation is different from~\eqref{DQ-2}.

\medskip

Formally, as $\ell\to 0$, the car density function $\rho$ 
satisfies the following  conservation law:
\begin{equation}\label{claw}
\rho_t + f(k(x),\rho)_x =0, \qquad \mbox{where}\quad 
f(k,\rho)\;\dot=\;   k \rho \, \phi(\rho). 
\end{equation}
Here $k(x)$ is discontinuous at $x=0$.
Two types of jumps occur in the solution, 
namely
the $k$-jump at $x=0$ and the $\rho$-shock where $k$ is constant.
The $\rho$-shock and its corresponding traveling wave profiles of the FtL model 
is studied in~\cite{ShenKarim2017}, where existence, uniqueness and local stability
are proved.
In this paper we consider the $k$-jump at $x=0$, 
and analyze the stationary profile $Q(x)$ that
connects the two constant states $\rho_\pm$ as $x\to\pm\infty$.   

There are various cases, with different relations between $(V_-,V_+)$ and $(\rho_-,\rho_+)$. 
For each of these cases, 
we study the initial value problem for~\eqref{DQ-2},
with initial data given on $x\ge 0$. 
Due to the discontinuity in the coefficient $k(x)$, the analysis is non-trivial. 
The initial value problem of the DDDE~\eqref{DQ-2} 
can be solved by method of steps, solving backwards in $x$ over a 
suitable interval in each step. 
At some steps, as $x$ or $x^\sharp$ cross 0, 
one needs to solve a discontinuous ODE.
The existence and well-posedness of the solutions can be established 
under the {\em transversality condition}, i.e., at every point
where the right hand side of the ODE has a jump,
the vector field for ODE crosses the curve of jump transversally. 
For literature on discontinuous ODEs and transversality condition, 
we refer to~\cites{MR964856,MR1634652,MR2239406,MR1028776}
and the references therein.

% key condition for the existence and well-posedness of the solutions
%is the so-called  transversality condition, where 
%
%of the vector field with respect to the curves of discontinuity, 
%which in term yields 
%the existence and well-posedness of the solutions for the initial value problem.

We also show that  the solution of the initial value problem 
with suitable initial data gives
the desired stationary profile $Q(x)$ with the given boundary conditions at
the infinities.
For different cases we prove that:
(i) there exist infinitely many profiles, 
(ii) there exists exactly one profile, or (iii) no profile exists. 
Depending on the case, 
some of the profiles attract nearby solutions for the FtL model,
while others are unstable.

\bigskip
We compare our result to the classical vanishing viscosity approach.
The conservation law~\eqref{claw} can be approximated by a viscous 
equation 
\bel{claw-vis}
\rho_t + f(k(x),\rho)_x = \ve \rho_{xx},
\eeq
where $\ve>0$ is a small parameter representing the viscosity. 
When $k(x)$ has a jump as in~\eqref{eq:V0}, the $k$-jump
at $x=0$ has a corresponding stationary viscous profile $\rho^\ve(x)$,
satisfying the ODE
\bel{ODE-vis}
 \frac{d}{dx} \rho^\ve(x) =\frac{1}{\ve} \left[ f(k(x),\rho^\ve(x)) - \bar f \; \right],
\qquad \mbox{where}\quad 
\bar f = f(V_-,\rho_-) = f(V_+, \rho_+).
\eeq

Monotone viscous profiles exist if one of the followings holds:
\begin{itemize}
\item We have $\rho_- < \rho_+$ and there exists a $\hat\rho \in[\rho_-,\rho_+]$ such that 
\[
f(V_-,\rho)>\bar f ~ \mbox{for}~\rho\in[\rho_-,\hat\rho], 
\quad\mbox{and}\quad 
f(V_+,\rho)>\bar f ~ \mbox{for}~ \rho\in [\hat\rho,\rho_+].
\]
\item We have $\rho_->\rho_+$ and there exists a $\hat\rho \in[\rho_+,\rho_-]$ such that 
\[
f(V_-,\rho)<\bar f ~ \mbox{for}~\rho\in[\rho_+,\hat\rho], 
\quad\mbox{and}\quad 
f(V_+,\rho)<\bar f ~ \mbox{for}~ \rho\in [\hat\rho,\rho_-].
\]
\end{itemize}
See~\cites{MR1109304,MR3663611,GS2016} for more details.
For other general references on scalar conservation law with discontinuous 
coefficient, we refer to a survey paper~\cite{MR3416038} and the references therein. 
Other related references on micro-macro models for traffic flow and their analysis
include~\cites{MR2727134, MR3253235,MR3541527,MR3177735}.
We would like to mention a recent work \cite{MR3714974}  
(and the references therein), 
which considers the traveling waves for degenerate diffusive equations on network,
where a necessary and sufficient algebraic condition  is established 
for the existence of traveling waves. 

\medskip

The rest of the paper is organized as follows. 
In section 2 we present various technical Lemmas, on specific properties 
for the solutions of~\eqref{DQ-2} and~\eqref{DW}. 
Section 3 is dedicated to the case with $V_->V_+$,
where 4 sub-cases are analyzed in detail. 
The analytical result is also confirmed by numerical simulations.
For one sub-case, we also show that the profiles $Q(x)$ are 
attractor for the solutions of the FtL model. 
The case with $V_-<V_+$ is studied in section 4,
following a similar line of approach as in section 3. 
The analysis for the main sub case here is much more involving
due to the lack of monotonicity. 
In section 5 we present a numerical simulation with ``Riemann initial data''. 
Finally, concluding remarks are given in section 6.

%%%%%%%%%%%%%%%%%%%
\section{Technique Lemmas}\setcounter{equation}{0}
For the rest of the paper, 
we denote the flux functions
\bel{fpm}
f_-(\rho)\; \dot=\; V_- \,\rho \,\phi(\rho),
\qquad 
f_+(\rho) \;\dot=\; V_+\,\rho \,\phi(\rho).
\eeq

Since the jump is stationary, 
the Rankine-Hugoniot condition requires
\begin{equation}\label{RH}
f_- (\rho_-) =f_+(\rho_+) \;\dot=\; \bar f\ge 0. 
 \end{equation}
We note  that the cases with  $\bar f=0$  are trivial, since they represent the cases
where the road is either empty or completely bumper-to-bumper. 
Indeed, we have:
\begin{itemize}
\item If $\rho_-=\rho_+=0$ then there is no car on the road; 
\item If $\rho_-=\rho_+=1$ then the road is completely bumper-to-bumper with cars and no one moves;
\item If $\rho_-=0,\rho_+=1$, then there is no car on $x<0$ but 
completely bumper-to-bumper  on $x>0$, therefore no one moves.
\end{itemize}
For the rest of the discussion, we assume 
\[
\bar f >0, \qquad \mbox{i.e.}~~ 0 <\rho<1.
\]

%Given a profile $Q(x)$ one can generate a sequence of car positions $\{z_i\}$.
We start with some definitions.
\begin{definition}\label{def:1}
Let $Q(x)$ be a continuous function defined on $x\in\mathbb{R}$ with $0 < Q(x) <1$. 
We call a sequence of car positions  $\{z_i\}$
\textbf{a distribution of car positions generated by $Q(x)$}, if 
\begin{equation}\label{eq:def1}
%\rho_i=Q(z_i), \qquad 
z_{i+1}-z_i = \frac{\ell}{Q(z_i)}, \qquad \forall i\in\mathbb{Z}.
\end{equation}
%Here $\rho_i$ is the discrete density for the $i$th car. 
\end{definition}

Note that if one imposes $z_0=0$, then the distribution  $\{z_i\}$ is unique.

\begin{definition}\label{def:2}
Given a profile $Q(x)$ and a distribution of car positions $\{z_i(t)\}$.
Let  $\{\rho_i(t)\}$ be the corresponding discrete densities for the cars, 
computed as~\eqref{def-rho-i}.
We say that $\{z_i(t),\rho_i(t)\}$ \textbf{traces along $Q(x)$},  if 
\[
 Q(z_i(t)) = \rho_i(t),  \qquad \forall i\in\mathbb{Z},~~ t\ge 0.
\]
\end{definition}

The following Lemma is  immediate.

\begin{lemma}\label{lm:0}
Let $Q(x)$ be a given profile and $\{z_i(0)\}$ be a distribution generated by $Q(x)$.  
Let $\{z_i(t)\}$ be the solution of~\eqref{dotzi} with initial data $\{z_i(0)\}$,
and let $\{\rho_i(t)\}$ be the corresponding discrete density. 
Then, $Q(x)$ satisfies~\eqref{DQ} 
if and only if $\{z_i(t),\rho_i(t)\}$ traces along $Q(x)$.
\end{lemma}

Solutions of~\eqref{DQ} exhibit a periodical behavior. 

\begin{lemma}\label{lm1} (\textbf{Periodicity})
Let a continuous function $Q(x)$ be given  on $x\in\mathbb{R}$ with $0<Q(x)<1$.
Let $\{z_i(0)\}$ be a distribution of car positions generated by $Q(x)$,
and let $\{z_i(t)\}$ be the solution of the FtL model~\eqref{dotzi} with this 
initial data. 
Then the followings are equivalent.
\begin{itemize}
\item[(a)] $Q(x)$ satisfies the equation~\eqref{DQ};
\item[(b)]  There exist a constant period $t_p$ such that 
\bel{eq:lm1}
 z_i(t+t_p) = z_{i+1}(t), \qquad \forall i\in\mathbb{Z}, t\ge 0.
 \eeq
\end{itemize}
\end{lemma} 

\begin{proof}
We first prove that (b) implies (a). 
%Indeed, since the system is autonomous, we can set $t=0$. 
Writing
\[ 
z_i(0)=x, \qquad z_{i+1}(0) = x^\sharp= x+ \ell/Q(x) , 
\]
and using 
\[ 
\frac{dz}{dt} = k(z) \cdot \phi(Q(z)), \qquad \rightarrow \quad 
\frac{dz}{ k(z) \cdot \phi(Q(z))} = dt,
\]
the time it takes for car no $i$ to reach the position of its leader is 
\[
 t_p = 
\int_{x}^{x+\ell/Q(x)} 
\frac{1}{k(z) \phi(Q(z))} dz = \mbox{constant}.
\]
Differentiating the above equation in $x$ on both sides, one gets
\[
 (1- \ell Q'(x)/Q^2(x)) \frac{1}{k(x^\sharp) \phi(Q(x^\sharp))}
-\frac{1}{k(x)\phi(Q(x))} =0,
\]
which easily leads to~\eqref{DQ}.
The proof for (a) implies (b) can be obtained by reversing the order of the above 
arguments.
\end{proof}

The next lemma connects the period $t_p$ with the flux $\bar f$ at the infinities.

\begin{lemma}\label{lm:2}
(i)   In the setting of Lemma~\ref{lm1}, if we have 
   \bel{eq:lm2-1} \lim_{x\to\infty} Q(x) = \rho_+, \quad 
   \lim_{x\to-\infty} Q(x) = \rho_-, \quad f_-(\rho_-)=f_+(\rho_+) =\bar f,
   \eeq
   then the period is determined as
   \bel{eq:lm2-2} t_p = \frac{\ell}{\bar f}.\eeq
(ii)  On the other hand, if the period $t_p$ is given and the
   solution approach some asymptotic limits such that
   \[ \lim_{x\to\infty} Q(x)=\rho_+, \quad \lim_{x\to-\infty} Q(x)=\rho_-,\]
   then the limits must satisfy
   \[
   f_-(\rho_-)=f_+(\rho_+) =\frac{\ell}{t_p}. %,\qquad \rho_- <\rho^*<\rho_+.
   \]
\end{lemma}

The proof is for Lemma~\ref{lm:2} is the same as the proof of
Lemma~2.7 in \cite{ShenKarim2017}. 
We skip the details.

\bigskip
Next Lemma shows that the solution $Q(x)$ is monotone 
in some sense of ``average''.

\begin{lemma}\label{lm4}
Let $Q(x)$ be a profile that satisfies~\eqref{DQ}. Given $x$,  we let 
\[x^\sharp=x+\ell/Q(x)\]
be the position of the leader for the car at $x$. 
Then, for any $x$, we have 
\bel{ss-a}
\frac{\ell}{\bar f} - \frac{\ell}{f(k(x),Q(x))} 
= 
 \int_x^{x^\sharp}
 \left[
\frac{1}{k(z) \phi(Q(z))} - \frac{1}{k(x)\phi(Q(x))}\right] \; dz.
\eeq
When $k(x)\equiv V$ is constant on $[x,x^\sharp]$, \eqref{ss-a} is simplified to 
\bel{ss5}
\frac{\ell}{\bar f} - \frac{\ell}{f(V,Q(x))} 
= \frac{1}{V}
 \int_x^{x^\sharp}
 \left[
\frac{1}{ \phi(Q(z))} - \frac{1}{\phi(Q(x))}\right] \; dz.
\eeq
\end{lemma}

\begin{proof}
The Lemma follows immediately from the periodicity property in Lemma~\ref{lm1}
\[
\frac{\ell}{\bar f} 
= 
\int_x^{x^\sharp}
\frac{1}{k(z) \phi(Q(z))} \; dz,
\]
and subtracting from it the identity
\[
\frac{\ell}{f(V,Q(x))}  =  \frac{1}{V}
 \int_x^{x^\sharp}\frac{1}{\phi(Q(x))} \; dz.
\] 
\end{proof}
\begin{remark}
Since $\phi'<0$, the mapping $\rho\mapsto (1/\phi(\rho))$ 
is monotone increasing. 
Then, \eqref{ss5} roughly 
says that if $f(V,Q(x)) > \bar f$ at some $x$, 
then some ``averaged-value'' of $Q$ on 
$[x,x^\sharp]$ 
is larger than $Q(x)$, so in ``average'' $Q(x)$ is increasing.
Similarly, if $f(V,Q(x)) < \bar f$ at some $x$, 
then  in ``average'' $Q(x)$ is decreasing.
\end{remark}

\begin{lemma}\label{lm4.5}
Let $Q(x)$ be a profile that satisfies~\eqref{DQ}. 
Let $\{z_i\}$ be  a distribution of car positions generated by $Q(x)$.
Then, for any $y$ with 
\[ z_i < y < z_{i+1}\]
we have 
\begin{equation}\label{eq:lm45}
 z_{i+1} < y^\sharp < z_{i+2}, \qquad \mbox{where} \quad 
y^\sharp = y+ \ell/Q(y).
\end{equation}
\end{lemma}

\begin{proof}
We prove by contradiction. 
%First we observe that $y^\sharp>y$.  
We first assume that
\[
y^\sharp \le z_{i+1}, \qquad \mbox{therefore}\quad [y,y^\sharp] \subset [z_i,z_{i+1}] .
\]
By the periodic property in Lemma~\ref{lm1}, we have
\[ %\bel{eq:lnn}
t_p = \int_{z_i}^{z_{i+1}} \frac{1}{k(z)\phi(Q(z))} dz  > 
\int_{y}^{y^\sharp} \frac{1}{k(z)\phi(Q(z))} dz=t_p,
\] %\eeq
a contradiction.  
We now assume 
\[
y^\sharp  \ge z_{i+2} \qquad \mbox{therefore}\quad 
[z_{i+1},z_{i+2}]\subset [y,y^\sharp]  .
\]
But again, the periodic property in Lemma~\ref{lm1} implies
\[
t_p = \int_{z_{i+1}}^{z_{i+2}} \frac{1}{k(z)\phi(Q(z))} dz  <
\int_{y}^{y^\sharp} \frac{1}{k(z)\phi(Q(z))} dz=t_p,
\]
again a contradiction.  
Thus, we conclude~\eqref{eq:lm45}, completing the proof. 
\end{proof}

\bigskip

We now establish the invariant regions $Q(x)>\rho_-$ and $Q(x)<\rho_-$,
on $x<0$.

\begin{lemma}\label{lm5}
Let $k(x)$ be the step function in~\eqref{eq:V0}, and let $Q(x)$ be a 
profile that satisfies~\eqref{DQ-2} with 
\[
\lim_{x\to\infty} Q(x) = \rho_+, \qquad  \mbox{where} 
\quad
%\rho_+>\rho^*, \quad f'(\rho^*)=0, \quad
\bar f = f_+(\rho_+).
\]
Let $\rho^*$ be the unique stagnation point where $f_-'(\rho^*)=0$, and 
$\rho_-<\rho^*$ be the value that satisfies $f_-(\rho_-)=\bar f$. 
Denote the interval 
\[ I=[y,y^\sharp] \qquad \mbox{where}\quad  y^\sharp=y+\ell/Q(y) \le 0.\] 
Then,  the followings hold.
\begin{itemize}
\item[(a)] 
If $f_-(Q(x)) > \bar f$ and $Q(x)>\rho_-$ for $x\in I$, then 
%$f_-(Q(x)) > \bar f$ and $Q(x)>\rho_-$ 
the same holds for all $x\le y$.
\item[(b)] 
If $f_-(Q(x)) <\bar f$ and $Q(x)<\rho_-$ for $x\in I$, then 
%$f_-(Q(x)) < \bar f$ 
the same holds
for all $x\le y$.
\end{itemize}
In both cases, we have 
\[ \lim_{x\to-\infty} Q(x) = \rho_-.\]
% \qquad \mbox{where} \quad \rho_-<\rho^*,\quad f'(\rho^*)=0,  
%~\mbox{and}~f_-(\rho_-)=\bar f.\]
\end{lemma}

\begin{proof}
We only prove (a), while the proof for (b) is  similar. 
The proof is achieved by contradiction. 
Suppose that $f_-(Q(x)) > \bar f$ and  $Q(x) > \rho_-$ on $x\in I$. 
%Then $Q(x) > \rho_-$ on $I$.
% where $\rho_-<\rho^*, f'(\rho^*)=0$ and $f_-(\rho_-)=\bar f$. 
First, we assume that  $Q(x)$ can be less than $\rho_-$  for $x\le y$. 
Let $\bar y$ be the right most point where $Q(x)$ crosses $\rho_-$,
such that
\bel{eq:v} 
Q(\bar y)=\rho_-, \qquad Q(x) > \rho_- \quad \mbox{for} \quad
x>\bar y.
\eeq
Now~\eqref{ss5} implies that the ``average'' value of $Q(x)$ on the 
interval $[\bar y, \bar y+\ell/Q(\bar y)]$ is $\rho_-$.
Clearly, this contradicts~\eqref{eq:v}.
On the other hand, we assume that $Q(x)$ can be bigger than 
$\hat \rho$ where $f_-(\hat \rho)=\bar f$ and $\hat \rho >\rho^*$. 
Let $\hat y$ be the right most point where $Q(x)$ crosses $\hat\rho$,
such that
\[
Q(\hat y)=\hat\rho, \qquad Q(x) < \hat\rho \quad \mbox{for} \quad
x>\hat y.
\]
Again, this contradicts~\eqref{ss5}, proving (a).

To prove the asymptotic limit, let $\{z_i\}$ be a distribution of car position generated
by $Q(x)$ with $z_0 = y^\sharp$, and denote the interval
$I_k=[z_k,z_{k+1}]$. 
Let
\[ M_k \;\dot=\; \max_{x\in I_k}  \frac{1}{\phi(Q(x))}, \quad  k\le -2,\]
and let $\{y_k\}$ be the points where these maxima are attained:
\[  \frac{1}{\phi(Q(y_k))} = M_k, \quad  k\le -2.\]
We claim that 
\begin{equation}\label{Mk}
M_{k+1} - M_{k} \ge \mathcal{O}(1) \cdot (Q(y_k)-\rho_-), \qquad \mbox{for}~ k<-2,
\end{equation}
which implies that 
\[ \lim_{k\to-\infty} M_k = \frac{1}{\phi(\rho_-)}, \qquad \mbox{and}\quad 
\lim_{x\to-\infty} Q(x) = \rho_-.\]
Indeed, if $Q(x)$ is monotone on $I_k$ for some $k\le -2$, 
then $Q(x)$  must be monotone increasing on $I_k$
due to~\eqref{ss5}.  
An induction argument shows that $Q(x)$ is monotone on $x\le z_k$.
Then
\[  y_k = z_{k+1}, \qquad  M_k= 1/\phi(Q(z_{k+1})).\]
Now, ~\eqref{ss5} gives
\[
\frac{\ell}{\bar f} -\frac{\ell}{f_-(Q(z_k))} \le \frac{z_{k+1}-z_k}{V^-} 
\cdot
\left[\frac{1}{\phi(Q(z_{k+1}))} - \frac{1}{\phi(Q(z_{k}))}\right]
=\frac{\ell (M_k-M_{k-1})}{V_-Q(z_k)},
\]
which implies
\[
M_{k}-M_{k-1} 
\ge V_-Q(z_k) \left(\frac{1}{f_-(\rho_-)} - \frac{1}{f_-(Q(z_k))}\right)
= \mathcal{O}(1) \cdot (Q(y_{k-1})-\rho_-).
\]

Now consider the case where $Q(x)$  is not monotone 
on any interval $I_k$, such that  $x\mapsto 1/\phi(Q(x))$ is oscillatory 
with at least one local minimum or local maximum  on any $I_k$
  for $k\le -2$ .  
Then, generically for some index $k<-2$,  
$M_k$ is attained at a local maximum of $1/\phi(Q(x))$, say $y_k\in I_k$.
Then $y_k$ is the local maximum of $Q(x)$  on $I_k$, with $Q'(y_k)=0$. 
Denoting its leader as $y^\sharp_k$, we have 
$y^\sharp_k \in I_{k+1}$ by Lemma~\ref{lm4.5}.
Also, $Q'(y_k)=0$ implies that  $Q(y_k)=Q(y_k^\sharp)$.
Then~\eqref{ss5} implies that there exists a local maximum 
$y'_{k+1} \in ( z_{k+1}, y^\sharp_k)$ with $Q(y'_{k+1})>Q(y_k)$.  
See Figure~\ref{fig:explain} for an illustration.

\begin{figure}[htbp]
\begin{center}
\setlength{\unitlength}{0.85mm}
\begin{picture}(60,38)(-3,-5) 
\multiput(0,0)(0,1){30}{\line(0,1){0.5}}\put(0,-4){$z_k$}
\multiput(25,0)(0,1){30}{\line(0,1){0.5}}\put(22,-4){$z_{k+1}$}
\multiput(55,0)(0,1){30}{\line(0,1){0.5}}\put(52,-4){$z_{k+2}$}
\multiput(5,20)(1,0){50}{\line(1,0){0.5}}
\multiput(10,5)(0,1){20}{\line(0,1){0.5}}\put(7,2){$y_k$}
\multiput(40.5,6)(0,1){25}{\line(0,1){0.5}}\put(37,1){$y'_{k+1}$}
\multiput(50.5,10)(0,1){13}{\line(0,1){0.5}}\put(48,4){$y^\sharp_{k}$}
\thicklines
\qbezier(0,10)(10,30)(20,10)
\qbezier(20,10)(25,2)(28,10)
\qbezier(28,10)(40,50)(55,10)
\end{picture}
\caption{Graph of $Q(x)$ on the interval $[z_k,z_{k+2}]$. 
Illustration of the locations for $y_k,y^\sharp_k$ and $y'_{k+1}$, used in the proof of Lemma~\ref{lm5}.}
\label{fig:explain}
\end{center}
\end{figure}
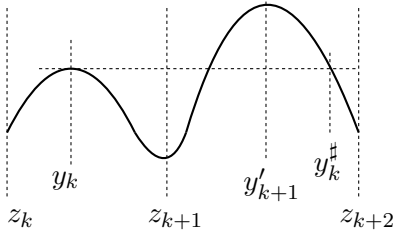

Furthermore, applying~\eqref{ss5} on $[y_k,y_k^\sharp]$ we get
\begin{eqnarray*}
%\frac{\ell}{V_-\rho_-\phi(\rho_-)} - \frac{\ell}{V_- Q(y_k)\phi(Q(y_k))} 
\frac{\ell}{f_-(\rho_-)} - \frac{\ell}{f_- (Q(y_k))} 
&=& \frac{1}{V_-}\int_{y_k}^{y_k^\sharp} \left[ \frac{1}{\phi(Q(z))} 
-\frac{1}{\phi(Q(y_k))}\right] dz\\
&<&   \frac{1}{V_-} \cdot \frac{\ell}{Q(y_k)} \cdot 
\left[\frac{1}{\phi(Q(y'_{k+1}))} 
-M_k \right].
\end{eqnarray*}
Since $M_{k+1} \ge \frac{1}{\phi(Q(y'_{k+1}))} $, this gives
\[ 
M_{k+1}-M_{k} 
>
V_- Q(y_k)  \left[\frac{1}{f_-(\rho_-)} - \frac{1}{ f_-(Q(y_k))}  \right]
=\mathcal{O}(1) \cdot \left[Q(y_{k})-\rho_-\right] ,
\]
% we get 
%\[
%
%\]
%\[
%Q(y'_{k+1})-Q(y_k) \ge 
%Q(y_k) \left(\frac{1}{\rho_- \phi(\rho_-)} - \frac{1}{Q(y_k)\phi(Q(y_k))}\right)
%\ge \mathcal{O}(1) \cdot (Q(y_k)-\rho_-).
%\]
%Since 
%\[ \ge \mathcal{O}(1) \cdot \left[Q(y_{k})-\rho_-\right]
%\ge \mathcal{O}(1) \cdot \left[Q(z_{k})-\rho_-\right] , \]
 completing the proof.
\end{proof}

\bigskip

\begin{lemma}\label{lm3}
(\textbf{Ordering of the profiles})
Assume that there exist multiple profiles that solve the equation~\eqref{DQ-2} 
with  asymptotes $\rho_\pm$ that satisfies~\eqref{RH}. 
Then the graphs of these profiles never intersect. 
\end{lemma} 

\begin{proof}
We prove by contradiction. 
Assume that there exist two profiles $Q_1(x),Q_2(x)$ 
which intersect at a point $y$, such that 
\[
 Q_1(y)=Q_2(y), \qquad  Q_1(x) > Q_2(x) ~~\mbox{for}~ x> y.
\]
Let 
\[y^\sharp \;\dot=\;y + \frac{\ell}{ Q_1(y)}=y + \frac{\ell}{ Q_2(y)}
\]
be the position of the leader 
for the car at $y$ for both profiles, and let $t_{p,1}$ and $t_{p,2}$ be the times
for the car at $y$ to reach its leader's position at $y^\sharp$, 
tracing along $Q_1(x)$ and $ Q_2(x)$, respectively. 
Then 
\[
 t_{p,1} = \int_{y}^{y^\sharp} \frac{1}{k(x) \phi(Q_1(x))} dx~ >~
\int_{y}^{y^\sharp} \frac{1}{k(x) \phi(Q_2(x))} dx = t_{p,2}.
\]
Since both profiles $Q_1,Q_2$ approach the same asymptotic limits, 
by Lemma~\ref{lm:2} one must have $t_{p,1}=t_{p,2}$,
a contradiction. 
\end{proof}

%%%%%%%%%%%%%%%%%%%%%%%%%%%%%--------------------%%%%%%
\section{Case 1: $V_->V_+$.}
\setcounter{equation}{0}
%%%%%%%%%%%%%%%%%%%%%%%%%%%%%--------------------%%%%%%

In this section we consider the case where the speed limit has a downward jump at $x=0$. 
Recall the Rankine-Hugoniot jump condition~\eqref{RH}. 
Fix a $\bar f$, with
\[
0<\bar f \le f_+(\rho^*), \qquad \mbox{where}\quad  f_-'(\rho^*)=f_+'(\rho^*)=0,
\]
and let $\rho_1^-,\rho_2^-,\rho_1^+,\rho_2^+$ be the unique values that satisfy
\bel{fbar}
f_-(\rho_1^-)=f_-(\rho_2^-)=f_+(\rho_1^+)=f_+(\rho_2^+)=\bar f,
%f_\pm(\rho_{1,2}^\pm)=\bar f, 
\quad \mbox{and}~
\rho_1^- < \rho_1^+ \le \rho^* \le \rho_2^+ < \rho_2^- .
\eeq
See Figure~\ref{fig:rhos} for an illustration. 
Note that we may have $\rho_1^+ = \rho^* =\rho_2^+$ when $\bar f = f_+(\rho^*)$.

\begin{figure}[htbp]
\begin{center}
\setlength{\unitlength}{0.85mm}
\begin{picture}(60,50)(-3,-5)  % -- Left ---
\put(0,0){\vector(1,0){60}}\put(60,-2){$\rho$}
\put(0,0){\vector(0,1){40}}
\multiput(0,15)(2,0){24}{\line(1,0){1}}\put(-5,13){$\bar f$}
\multiput(5.2,15)(0,-2){8}{\line(0,-1){1}}\put(2,-4){ $\rho^-_1$}
\multiput(12.5,15)(0,-2){8}{\line(0,-1){1}}\put(11,-4){ $\rho_1^+$}
\multiput(25,40)(0,-2){20}{\line(0,-1){1}}\put(22,-4){ $\rho^*$}
\multiput(37.5,15)(0,-2){8}{\line(0,-1){1}}\put(33,-4){ $\rho^+_2$}
\multiput(44.7,15)(0,-2){8}{\line(0,-1){1}}\put(43,-4){ $\rho^-_2$}
\put(25.5,22){$f_+$}\put(25.5,36){$f_-$}
\thicklines
\qbezier(0,0)(25,40)(50,0)
\qbezier(0,0)(25,80)(50,0)
\end{picture}
\caption{Flux functions $f_-, f_+$, and the locations of 
$\rho^-_1,\rho^+_1,\rho^-_2,\rho^+_2$, and $\rho^*$.}
%which are candidates for $\rho_-,\rho_+$.}
\label{fig:rhos}
\end{center}
\end{figure}
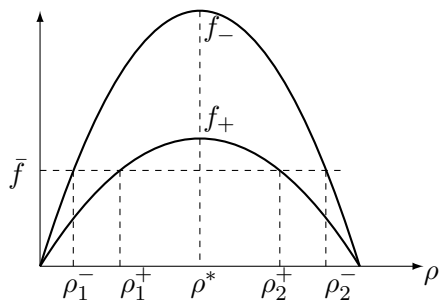

There are 4 possible combinations of $(\rho_-,\rho_+)$ which satisfy~\eqref{fbar}:
\begin{itemize}
\item[1A.] $(\rho_-,\rho_+)=(\rho_1^-,\rho_2^+)$, i.e., 
$0<\rho_- < \rho^* < \rho_+<1$;
\item[1B.] $(\rho_-,\rho_+)=(\rho_1^-,\rho_1^+)$, i.e., 
$0<\rho_- < \rho_+ \le \rho^*$;
\item[1C.] $(\rho_-,\rho_+)=(\rho_2^-,\rho_2^+)$, i.e., $ \rho^* < \rho_+<\rho_-<1$;
\item[1D.] $(\rho_-,\rho_+)=(\rho_2^-,\rho_1^+)$, i.e., $ 0<\rho_+\le\rho^* < \rho_-<1$.
\end{itemize} 

We denote by $W(x)$ the unique stationary profile that satisfies~\eqref{DW}, 
with 
\bel{W+}
   W(0)=\rho^*, \quad 
   \lim_{x\to-\infty} W(x) = \rho_1^+,
   \quad \lim_{x\to+\infty} W(x) = \rho_2^+.
\eeq
Note that any horizontal shifts of $W(x)$ is again a solution of~\eqref{DW}.
The existence and uniqueness of such a profile is proved
in~\cite{ShenKarim2017}.

We also recall Lemma~2.5 in~\cite{ShenKarim2017}, 
where the following is proved:
\begin{itemize}
\item As  $x\to\infty$,  $Q(x)$ can approach $\rho_+$ 
asymptotically with exponential rate 
only if $\rho_+ > \rho^*$. This means, if $\rho_+ \le \rho^*$, the
asymptote is unstable.
\item As  $x\to-\infty$,  $Q(x)$ can approach $\rho_-$ 
asymptotically with exponential rate 
only if $\rho_- < \rho^*$. This means, if $\rho_- \ge \rho^*$, the
asymptote is unstable.
\end{itemize}
%Here $\rho^*$  is the unique value where $f'(\rho^*)=0$. 

We discuss each sub-case in  detail in the rest of this section.

%%%%%%%%%%%%%%
\subsection{Case 1A:  $0 <\rho_- < \rho^* <  \rho_+<1$.}
%%%%%%%%%%%%%%

%\textbf{Initial value problems.} 
Since here $\rho_+>\rho^*$ is a stable asymptote,   on $x>0$ the solution for
$Q(x)$ must be either some horizontal shift of $W(x)$
or the trivial solution $Q(x) \equiv \rho_+$. 
For different horizontal shifts, these profiles take different values of $Q(0)$.
In all cases, we have 
\[
 \rho_1^+ < Q(0) \le \rho_+.
 \]
 
%%%%%%%%%%%
\subsubsection{The  initial value problems.}
%%%%%%%%%%%%%%
Once $Q(x)$ is given for $x\ge 0$, one can solve~\eqref{DQ-2} 
backward in $x$ as an ``initial value problem''. 
It is understood that the derivative in~\eqref{DQ-2} is the left derivative,
as one solves the equations backward in $x$.
The profile $Q(x)$, if exists,  can have kinks, but remains continuous. 
Next Theorem provides well-posedness of this initial value problem.  

%%%%%%%%%%%%
\begin{theorem}\label{tm:IC} (\textbf{Well-posedness of the initial value problems})
Let $V_->V_+$. 
Given $\rho_+$ such that $\rho^*<\rho_+ <1$.  
Consider the initial value problem for ~\eqref{DQ-2},
where an initial data is given on $x\ge0$, 
as either  a horizontal shift of $W(x)$ 
or the constant function $\rho_+$. 
Then, the initial value problem 
has a unique monotone solution $Q(x)$ on $x<0$.
\end{theorem} 
%%%%%%%%%%%%

\begin{proof} 
The proof takes a couple of steps.

\textbf{Step 1.}  
In the $(x,Q)$ plane, let $\mathcal{C}_0$ be the vertical line where  $x=0$,
and let $\mathcal{C}_1$  be 
the graph of the function $h(x) = -\ell/x$, for $x<-\ell$. 
The curve $\mathcal{C}_1$  indicates the position and local density 
of the cars whose leader is at $x=0$. Since the car length is $\ell$, 
the position of these car must be less than $-\ell$, so $h(x)$ is only defined 
on $x<-\ell$.
The discontinuities in~\eqref{DQ-2} occur  along $\mathcal{C}_0$ and 
$\mathcal{C}_1$. 
To ensure the existence and uniqueness of solutions, 
we must verify that the vector field of the DDDE~\eqref{DQ-2} 
must cross the curves of discontinuity {\em transversally}, see~\cite{MR964856}. 

Along $\mathcal{C}_0$, the discontinuity line is vertical, with infinite tangent. 
Thus, we need that
\begin{equation}\label{eq:cc1}
Q'(0\pm) ~ \mbox{is bounded}. 
\end{equation}
This is easily verified from~\eqref{DQ-2},  since $Q(0)\le \rho_+ <1$ 
so  $\phi(Q(0)) >0$. 

Along the curve $\mathcal{C}_1$, the tangent at a point $(x,h(x))$ is
\[
h'(x) = \ell/x^2 = h(x)^2/\ell.
\]
Let $Q(x)$ be a profile that solves~\eqref{DQ-2}, 
and let $y <0 $ be its intersection point with $\mathcal{C}_1$ such that
$Q(y) = h(y)$. It suffices to show that
\bel{eq:cc} Q'(y\pm) <  h'(y).\eeq
Indeed, from~\eqref{DQ-2} we have
\begin{eqnarray*}
Q'(y-) &
=& \frac{h(y)^2}{\ell \cdot \phi(h(y))} \left[\phi(h(y)) - \phi(Q(0))\right]
~=~h'(y) \left[1-\frac{\phi(Q(0))}{\phi(h(y)) }\right]  ,\\
Q'(y+) &
=& \frac{h(y)^2}{\ell V_-\phi(h(y))} \left[V_-\phi(h(y)) - V_+\phi(Q(0))\right]
~=~h'(y)\left[1-\frac{V_-\phi(Q(0))}{V_+\phi(h(y)) }\right] .
\end{eqnarray*}
Thus~\eqref{eq:cc} holds since $Q(0)<1$ and $\phi(Q(0))>0$.

\medskip

\textbf{Step 2.}  Once the transversality properties~\eqref{eq:cc1}-\eqref{eq:cc}
 are established, 
the existence and uniqueness of the solution for $Q(x)$ is achieved 
by method of steps. 
Denote 
\[I_k = [-k \ell , -(k-1)\ell], \qquad  \mbox{for}~~ k=1,2,3,\cdots.\] 
Consider $I_1$.  If $x\in I_1$, then its leader $x^\sharp$ is located at
\[
 x^\sharp= x+ \ell/Q(x) > 0.
 \]
We have an ODE with discontinuous right hand side, with
\bel{eq:cd}
  Q'(x) = \frac{Q(x)^2}{\ell \cdot V_- \phi(Q(x))} 
   \left[  V_- \phi(Q(x)) - V_+\phi(Q(x^\sharp)) \right]
\eeq
where $Q(x^\sharp)$ is given by the initial data on $x\ge0$.
Standard theory for discontinuous ODEs  (see~\cite{MR964856})
gives a uniqueness solution on $I_1$, provided that $Q(x)$ 
satisfies $ 0 < Q(x) < 1$ on $I_1$. 
Indeed, the lower bound $Q(x)> 0$ is a consequence of the fact that
$0$ is a critical point. Assuming that $Q(x)$ becomes negative on some 
subset of $I_1$, then there exists a point $\hat x \in I_1$ such that
$Q(\hat x)=0$ and $Q'(\hat x)>0$. 
But this is not possible because by~\eqref{DQ-2} we have
\[Q'(\hat x) = \frac{Q^2(\hat x)}{\ell \phi(Q(\hat x))} \left[\phi(Q(\hat x))
-\phi(\rho_+) \right]=0, \quad \mbox{where}~\rho_+=\lim_{x\to\infty}Q(x).
 \]
To prove the upper bound, we claim that 
$Q'(x) >0$ on $I_1$.  We argue with contradiction. 
Assuming that $Q(x)$ is not monotone on $I_1$, then there exists a point
$y \in I_1$ such that 
\[ 
Q'(y)=0, \qquad  Q'(x) \ge 0 \quad\mbox{for}~~ x>y.
\]
Since $Q'(0-)>0$, then $y <0$, and we have 
\bel{eq:ce}
  Q(y) < Q(y^\sharp),\qquad y^\sharp =y+\ell/Q(y) >0.  
\eeq
Now~\eqref{eq:cd} and $Q'(y)=0$ imply
\[
V_- \phi(Q(y)) - V_+\phi(Q(y^\sharp)) =0.
\]
Since $V_->V_+$ and $\phi'<0$, we get 
\[Q(y)>Q(y^\sharp),\]
a contradiction  to~\eqref{eq:ce}. 

\medskip

\textbf{Step 3.}
We iterate the argument in Step 2 for $k=2,3,\cdots$, 
until $I_k$ crosses the curve $\mathcal{C}_1$. 
After that,~\eqref{eq:cd} is replaced by
\bel{eq:cd3}
   Q'(x) = \frac{Q(x)^2}{\ell \cdot \phi(Q(x))} 
   \left[  \phi(Q(x)) - \phi(Q(x^\sharp)) \right],\qquad 
   x^\sharp=x+\ell/Q(x)<0.
\eeq
The same argument follows.
This proves the existence and uniqueness of a monotone solution
$Q(x)$ on $x<0$, for the initial value problem.
\end{proof}

\bigskip

%%%%%%%%%%%%%%%%%%%
\subsubsection{The boundary value problems.} 
%%%%%%%%%%%%%%%%%%%
Next Corollary establishes the existence of infinitely 
many monotone profiles $Q(x)$ for the boundary value problem,
with given boundary conditions  $\rho_-$ and $\rho_+$
at $ \pm\infty$.

\begin{corollary}\label{cor:1}
Let 
\[
V_->V_+,\qquad
0 <\rho_- \le \rho^* \le \rho_+<1, \qquad 
f_-(\rho_-)=f_+(\rho_+).
\]
There exist infinitely many monotone  profiles $Q(x)$
which satisfy the DDDE~\eqref{DQ-2},
and the boundary conditions
\bel{eq:ll}\lim_{x\to-\infty} Q(x) = \rho_-,\qquad \lim_{x\to+\infty} Q(x) = \rho_+.\eeq
Moreover, these profiles never intersect with each other, and
\bel{eq:Q0} \rho_1^+ < Q(0)\le\rho_+.\eeq
\end{corollary}

\begin{proof}
In Theorem~\ref{tm:IC} we show that there exist many profiles $Q(x)$ 
that satisfy~\eqref{DQ-2},
\eqref{eq:Q0}, 
and the second boundary condition in~\eqref{eq:ll}. 
Let $Q(x)$ be such a profile.  
It remains to show that the first boundary condition in~\eqref{eq:ll} holds.
Since $Q(x)$ is monotone and bounded below by $0$, 
then there exists an asymptotic limit
as $x\to-\infty$. 
Since $\lim_{x\to\infty} Q(x)=\rho_+$, 
by part (i) of Lemma~\ref{lm:2} the period must be 
\[t_p=\frac{\ell}{\bar f}, \quad \mbox{where} \quad \bar f = f^+(\rho_+).\]
By part (ii) of Lemma~\ref{lm:2} the limit at $x\to-\infty$ must be 
$\rho_-$ which satisfies 
$f^-(\rho_-)=\bar f$.
Since $\rho_-$ must a stable asymptote, we have $\rho_- \le\rho^*$.
%this limit must be $\rho_-$. 

The non-intersecting property of the profiles 
follows from Lemma~\ref{lm3}. 
\end{proof}

Sample profiles of $Q(x)$ with various $Q(0)$ values 
  are illustrated in Figure~\ref{fig:1A} plot (2), 
using 
\[
V_-=2, \quad V_+=1, \quad \ell=0.2, \quad \phi(\rho)=1-\rho,  \quad \bar f = 3/16.
\]

As comparison, 
we also illustrate the stationary viscous profiles.
For this sub-case there exist infinitely many stationary monotone viscous profiles
that satisfy the ODE~\eqref{ODE-vis}.
For each value of $\rho^\varepsilon(0) \in (\rho_1^+,\rho_+]$, there exists a unique
viscous profile. 
Sample viscous profiles $\rho^\varepsilon(x)$
with 
$\varepsilon=0.2$ and 
with various $\rho^\varepsilon(0)$ values 
are given in Figure~\ref{fig:1A} plot (3).

\bigskip

\begin{figure}[htb]
\begin{center}
\begin{tabular}{cc}
(1) & (2) \\
\setlength{\unitlength}{0.85mm}
\begin{picture}(83,50)(-10,-5)  
\put(0,0){\vector(1,0){60}}
\put(0,0){\vector(0,1){40}}
\multiput(5,15)(2,0){17}{\line(1,0){1}}
\multiput(37.5,15)(0,-2){8}{\line(0,-1){1}}\put(34,-3){ $\rho_+$}
\multiput(5,15)(0,-2){8}{\line(0,-1){1}}\put(2,-3){ $\rho_-$}
\multiput(12.5,15)(0,-2){8}{\line(0,-1){1}}\put(10,-4){ $\rho_1^+$}
\put(25,22){$f_+$}\put(25,36){$f_-$}
\put(5,15){\circle*{1.2}}\put(37.5,15){\circle*{1.2}}
\thicklines
\qbezier(0,0)(25,40)(50,0)
\qbezier(0,0)(25,80)(50,0)
\end{picture}
&
\includegraphics[width=6.5cm,clip,trim=0mm 0mm 5mm 5mm]{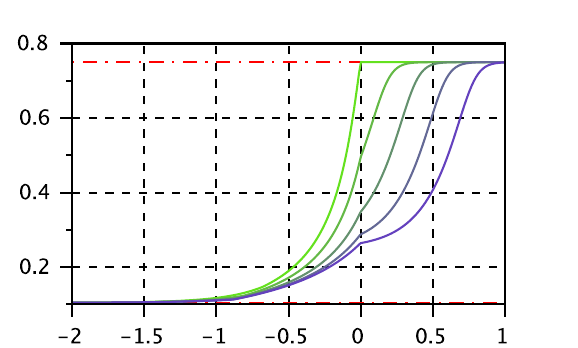}\\[5mm]
(3) & (4) \\
\includegraphics[width=6.5cm,clip,trim=0mm 0mm 5mm 5mm]{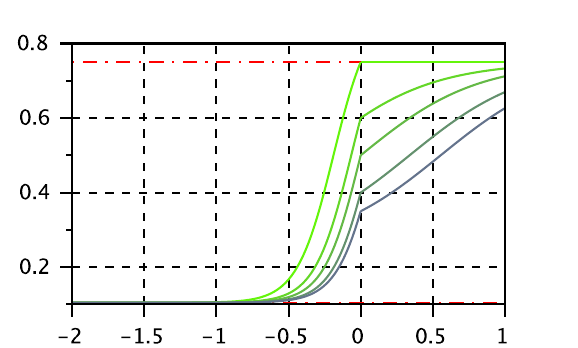} &
\includegraphics[width=6.5cm,clip,trim=0mm 0mm 5mm 5mm]{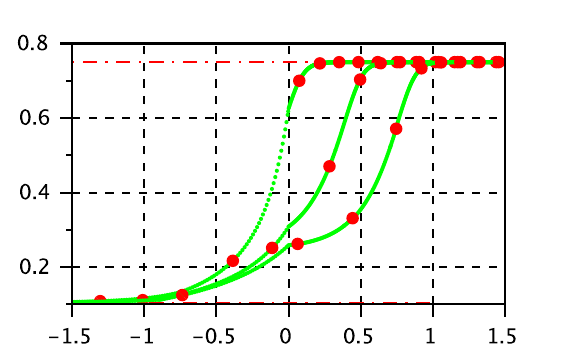}
\end{tabular}
\caption{Case 1A:  (1) Plot of the flux functions $f_-,f_+$ 
and the locations of $\rho_-,\rho_+$; 
(2) Plots of various profiles of $Q(x)$, with different values of $Q(0)$; 
(3) Plots of various viscous traveling waves $\rho^\ve(x)$, with different values
of $\rho^\ve(0)$; 
(4) Plots of various solutions of the FtL model $\{z_i(t), \rho_i(t)\}$,  with 3 different initial Riemann data. 
Here the thick dots denote the locations of cars at $t=2$.} 
\label{fig:1A}
\end{center}
\end{figure}
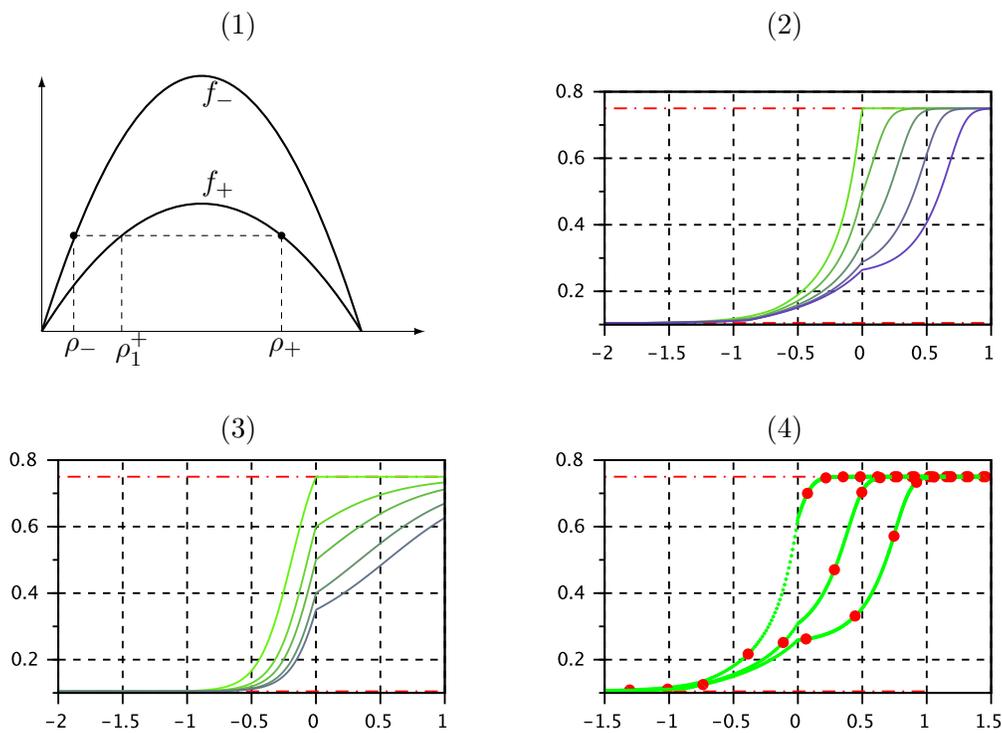

%%%%%%%%%%%%%%%%%%%%%%%%%
\subsubsection{Local stability of the profiles.} 
%%%%%%%%%%%%%%%%%%%%%%%%
We have shown that for each given $Q(0)\in(\rho_1^+, \rho_+]$, 
there exists a unique stationary profile $Q(x)$. 
Let $Q^\sharp(x)$ be the profile with $ Q^\sharp(0) =\rho_+$, 
and let $Q^\flat(x)$ be the limit profile as
$Q(0) \to \rho_1^+$. 
We define the domain
\bel{def:D}
D \;\dot=\; \left\{ (x,y)~ :~  Q^\flat(x) < y \le Q^\sharp(x),~  x\in\mathbb{R}\right\}.
\eeq
Clearly all profiles of $Q(x)$ lie in $D$. 
We now show that $D$ is a basin of attraction of the solution of the FtL,
in the sense described below.

Since all the profiles in $D$ never cross each other, 
we can parametrize the family of profiles, say by the value $Q(0)$. 
By continuity, 
any point $( x, y)\in D$ belong to a unique profile, call it $Q_{( x, y)}$ such that 
\[
Q_{( x, y)} (x) = y.
\]
For any point $(x,y)\in D$, we define the function
\bel{eq:defPsi}
  \Psi(x,y) \; \dot=\; Q_{(x,y)}(0),\qquad (x,y)\in D. 
\eeq

\begin{theorem}\label{tm:att}
Consider the setting of Corollary~\ref{cor:1} and let $D$ be defined as in~\eqref{def:D}.
Let $\{z_i(0)\}$ be a set of initial car positions and $\{\rho_i(0)\}$
be the corresponding discrete density defined as~\eqref{def-rho-i},
and assume that 
\bel{eq:assume}
\left(z_i(0), \rho_i(0)\right) \in D, \qquad \forall i\in\mathbb{Z}.
\eeq 
Let $\{z_i(t)\}$
be the solution of the FtL model with this initial data, 
and let $\{\rho_i(t)\}$ be the corresponding discrete density. 
Then 
\bel{eq:zt}
(z_i(t), \rho_i(t)) \in D, \qquad \forall t>0, ~\forall i\in\mathbb{Z}.
\eeq 
Denote
\[
 \Psi_i(t) = \Psi(z_i(t),\rho_i(t)), \qquad  i\in\mathbb{Z},
 \]
and define the total variation 
\[
\mbox{TV} \{\Psi_i(t)\}   \;\dot=\;  \sum_i \Big| \Psi_i(t)-\Psi_{i+1}(t)\Big|. 
\]
Then, we have
\bel{eq:tv} \lim_{t \to \infty} \mbox{TV} \{\Psi_i(t)\} =0,\qquad 
\mbox{i.e.,}\quad 
\lim_{t \to \infty} \Psi_i(t)=\tilde\Psi, \quad \forall i\in\mathbb{Z}.
\eeq
Thus, asymptotically the points $\{z_i(t),\rho_i(t)\}$  trace along 
 the profile $Q(x)$ with $Q(0) = \tilde\Psi$ as $t\to\infty$.
 \end{theorem}

\begin{proof}
We first assume~\eqref{eq:zt} and prove~\eqref{eq:tv}.  
Fix a time $\tau\ge0$. 
It suffices to show the followings: 
\begin{itemize}
\item[(i)] If $\Psi_m(\tau) > \Psi_{m+1}(\tau)$ at time $\tau$ for some $m$, 
then 
$ \frac{d}{dt} \Psi_m(\tau) <0 $; and
\item[(ii)]
If $\Psi_n(\tau) < \Psi_{n+1}(\tau)$ at time $\tau$ for some $n$, 
then 
$ \frac{d}{dt} \Psi_n(\tau) >0 $.
\end{itemize}
We prove (i) while (ii) can be proved in an entirely similar way. 
Let $\hat Q(x)$ be the profile that passes through the point 
$\{z_m(\tau),\rho_m(\tau)\}$. 
By the assumption $\Psi_m(\tau) > \Psi_{m+1}(\tau)$, and the point
$\{z_{m+1}(\tau),\rho_{m+1}(\tau)\}$  lies below the profile $\hat Q(x)$,
i.e.,
\bel{eq:Q1}
 \rho_{m+1}(\tau) <\hat Q(z_{m+1}(\tau)) .
\eeq
It suffices to show that 
\bel{eq:Q2}
\frac{\dot \rho_m(\tau)}{\dot z_{m}(\tau)} <\hat Q'(z_{m}(\tau)),
\eeq
%This is achieved with the same proof as for~\eqref{eq:Q2}, 
%replacing $Q^\sharp$ with $\hat Q$ and $i$ with $m$. 
indicating that the point $(z_m(\tau),\rho_m(\tau))$  moves below the profile 
$\hat Q(x)$ as $t$ increases from $\tau$. 
Indeed, equation~\eqref{DQ} gives
\bel{eq:Q3}
\hat Q'(z_{m}) =\frac{\hat Q^2(z_m)}{\ell k(z_m) \phi(\hat Q(z_m))}
\left[ k(z_m)\phi(\hat Q(z_m)) -  k(z_{m+1})\phi(\hat Q(z_{m+1})) 
\right] .
\eeq
On the other hand, \eqref{dotzi} and~\eqref{dot-rho-i} give
\bel{eq:Q4}
\frac{\dot \rho_m(\tau)}{\dot z_{m}(\tau)}
 =\frac{\rho_m^2}{\ell k(z_m) \phi(\rho_m)}
\left[ k(z_m)\phi(\rho_m) -  k(z_{m+1})\phi(\rho_{m+1})
\right] .
\eeq
Since $\rho_m=\hat Q(z_m)$, together with~\eqref{eq:Q1}, 
we conclude~\eqref{eq:Q2}.

We now prove~\eqref{eq:zt}. We consider the upper bound $Q^\sharp$, 
while the lower bound is entirely similar. 
Given a time $\tau\ge0$, we assume that $(z_i(\tau),\rho_i(\tau))\in D$ 
for all $i$, such that 
\[\rho_i(\tau) = Q^\sharp(z_i(\tau)),\qquad \forall i. \]
It suffices to show that, if there exist an index $m$ such that, 
\[
\rho_m(\tau) = Q^\sharp(z_m(\tau)), \quad 
\rho_{m+1}(\tau) \le Q^\sharp(z_{m+1}(\tau)), 
\]
then
\begin{equation}\label{eq:Q2b}
\frac{\dot\rho_m(\tau)}{ \dot z_m(\tau)} 
\le (Q^\sharp)'(z_m(\tau)), 
\end{equation}
The proof for~\eqref{eq:Q2b} is entirely similar to that of~\eqref{eq:Q2}, 
replacing $Q^\sharp$ with $\hat Q$. 
%
%
%Indeed, equation~\eqref{DQ} gives
%\bel{eq:Q3}
%( Q^\sharp)'(z_{i}) =\frac{ (Q^\sharp)^2(z_i)}{\ell k(z_i) \phi( Q^\sharp(z_i))}
%\left[ k(z_i)\phi(Q^\sharp(z_i)) -  k(z_{i+1})\phi( Q^\sharp(z_{i+1})) 
%\right] .
%\eeq
%On the other hand, \eqref{dotzi} and~\eqref{dot-rho-i} give
%\bel{eq:Q4}
%\frac{\dot \rho_i(\tau)}{\dot z_{i}(\tau)}
% =\frac{\rho_i^2(\tau)}{\ell k(z_i(\tau)) \phi(\rho_i(\tau))}
%\left[ k(z_i(\tau))\phi(\rho_i(\tau)) -  k(z_{i+1}(\tau))\phi(\rho_{i+1}(\tau))
%\right] .
%% =\frac{\rho_i^2}{\ell k(z_i) \phi(\rho_i)}
%%\left[ k(z_i)\phi(\rho_i) -  k(z_{i+1})\phi(\rho_{i+1})
%%\right] .
%\eeq
%Since $\rho_i(\tau)=Q^\sharp(z_i(\tau))$, together with~\eqref{eq:Q1}, 
%we conclude~\eqref{eq:Q2}.
%
\end{proof}

Numerical approximations are computed  for the solutions of the FtL model 
with the following ``Riemann initial data'', 
\bel{eq:rmzr}
z_i(0) = \begin{cases}
i \ell/\rho_+,  \quad & i\ge x_0, \\
i \ell/\rho_-,  \quad & i< x_0 ,
\end{cases} \qquad\qquad
\rho_i(0) = \begin{cases}
\rho_+,  \quad & i\ge x_0, \\
\rho_-,  \quad & i< x_0 .
\end{cases}
\eeq
The simulations are carried out for $0\le t \le 2$.
In Figure~\ref{fig:1A} plot (4), 
we plot the trajectory of $z_i(t)$ (in green) for 
the last period 
\[2-\frac{\ell}{\bar f}  \le t\le 2,\]
together with the car positions at $t=2$ as thick dots (in red).
The 3 profiles in the plot are for 
\[x_0=0, \quad x_0=0.3 \ell/\rho_-, \quad \mbox{and}\quad x_0=0.6 \ell/\rho_-.\]
Even though the initial data points $\{z_i(0),\rho_i(0)\}$ 
are not entirely in $D$, nevertheless we observe that the solutions of FtL model
converge quickly to certain profiles of $Q(x)$,
suggesting that Theorem~\ref{tm:att} probably applies to a larger domain.

All numerical simulations in this paper are carried out using SciLab. 
The source codes are available from the author's web-site~\cite{ShenWeb}.

%%%%%%%%%%%%%%%%%%%%%%%%
\subsection{Case 1B: $0<\rho_- < \rho_+ \le \rho^*$.}

Since $\rho^+\le \rho^*$ is an unstable asymptote for $x\to+\infty$,
the only solution on $x\ge 0$ is the constant solution $Q(x) \equiv \rho_+$.
Once $Q(x)$ is given on $x>0$, 
the rest can be solved backward in $x$ using~\eqref{DQ-2},
as an initial value problem. 
The existence and uniqueness of the profile follows from 
the same arguments as those
for Theorem~\ref{tm:IC} and Corollary~\ref{cor:1}. 
We summarize the result in next Theorem.

\begin{theorem}\label{tm:1B}
Let $V_->V_+$ and $0<\rho_- < \rho_+ \le \rho^*$ with $f_-(\rho_-)=f_+(\rho_+)$.
There exists a unique monotone profile $Q(x)$ which satisfies the equation~\eqref{DQ-2}
with
\[
 Q(x)=\rho_+ \quad \mbox{for}\quad x\ge 0, \qquad 
\lim_{x\to-\infty}Q(x) = \rho_-. 
\]
\end{theorem}

A typical plot of $Q(x)$ is given in Figure~\ref{fig:1B} plot (2).
%where the dashed red line is the value of $\rho_-$. 
%
As comparison, we also plot the viscous profile $\rho^\varepsilon(x)$
in Figure~\ref{fig:1B} plot (3),
with $\rho^\varepsilon(x)=\rho_+$ on $x\ge 0$. 
This is the only viscous profile that connects the two limit values $\rho_\pm$
at $x\to\pm\infty$. 

\textbf{Instability.} 
Since $\rho_+$ is an unstable asymptote, the profile is  unstable
with respect to perturbations on $x>0$, and the solution of the FtL model
can not converge to the profile in the sense of Theorem~\ref{tm:att}. 
Even if one starts with ``Riemann'' initial data with 
$\rho_i(0)=\rho_+$ for all  $z_i(0)\ge 0$, the perturbation, initially on $x<0$,
will propagate into the region $x>0$.  
Numerical simulation verifies this fact, see Figure~\ref{fig:1B} plot (4),
where a perturbation is formed and moves into $x>0$. 
Although on $x<0$ the FtL solution gets very close to the 
profile $Q$, the stability can not be achieved on $x>0$. 
This forward propagating wave  is caused by the fact that the characteristic speed
satisfies
\[
f'_-(\rho_-)>0,\qquad f'_+(\rho_-)>0,
\]
therefore information travels to the right.

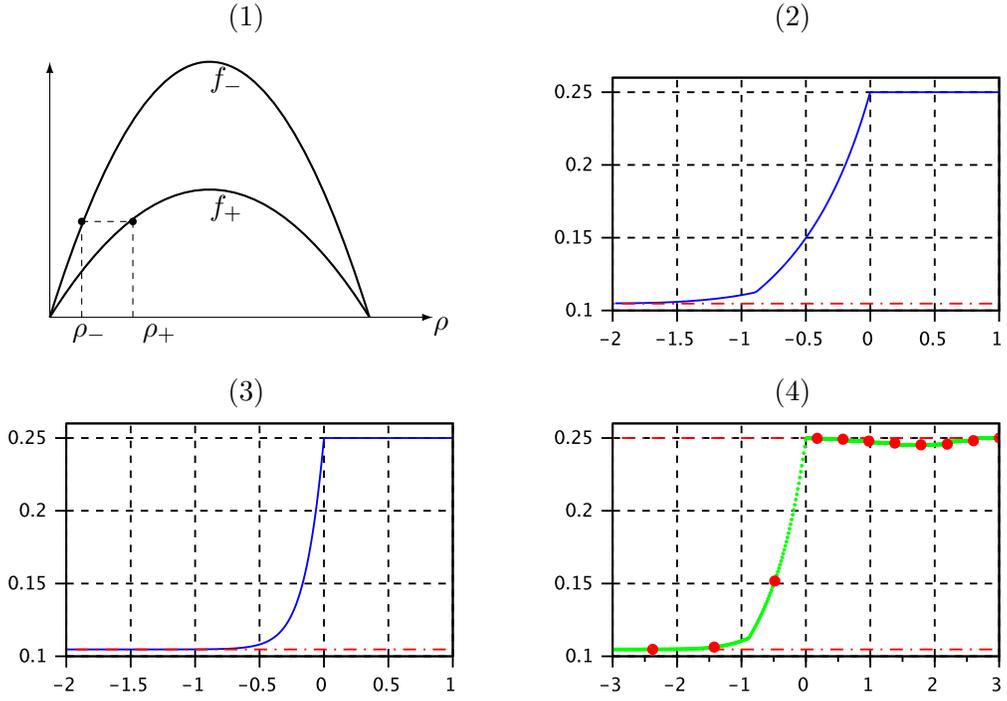
\begin{figure}[htbp]
\begin{center}
\begin{tabular}{cc}
(1) & (2) \\
\setlength{\unitlength}{0.85mm}
\begin{picture}(83,49)(-10,-5)  
\put(0,0){\vector(1,0){60}}\put(60,-2){$\rho$}
\put(0,0){\vector(0,1){40}}
\multiput(5,15)(2,0){4}{\line(1,0){1}}
\multiput(5,15)(0,-2){8}{\line(0,-1){1}}\put(2,-3){ $\rho_-$}
\multiput(13,15)(0,-2){8}{\line(0,-1){1}}\put(13,-3){ $\rho_+$}
\put(25,16){$f_+$}\put(25,36){$f_-$}
\put(5,15){\circle*{1.2}}\put(13,15){\circle*{1.2}}
\thicklines
\qbezier(0,0)(25,40)(50,0)
\qbezier(0,0)(25,80)(50,0)
\end{picture}
&
\includegraphics[width=6.5cm,clip,trim=0mm 0mm 5mm 5mm]{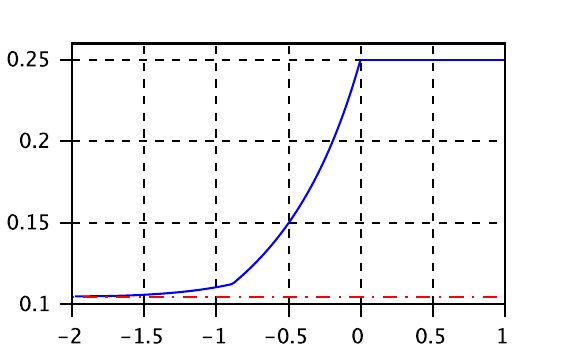}\\[2mm]
(3) & (4) \\
\includegraphics[width=6.5cm,clip,trim=0mm 0mm 5mm 5mm]{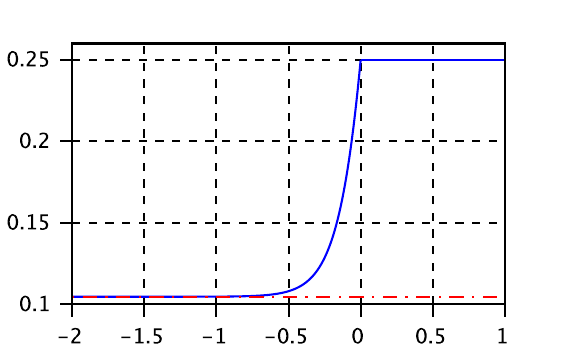}&
\includegraphics[width=6.5cm,clip,trim=0mm 0mm 5mm 5mm]{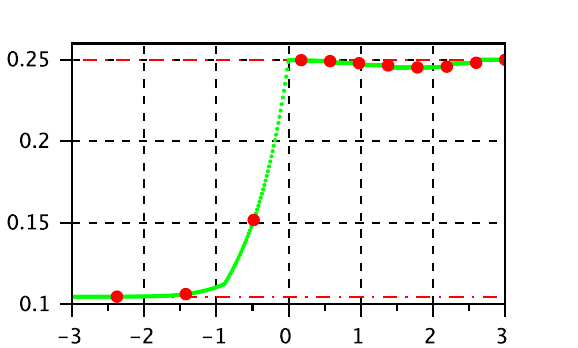}
\end{tabular}
\caption{Case 1B.  (1) Plots of the flux functions and the locations of $\rho_-,\rho_+$; 
(2) Plot of the unique stationary profile $Q(x)$ with $Q(0)=\rho_+$;
(3) Plot of the unique viscous profile $\rho^\ve(x)$ with $\rho^\ve(0)=\rho_+$; (4) Plot of the solution of the FtL model $\{z_i(t), \rho_i(t)\}$ with  a 
Riemann initial data. Here the thick dots are the locations
of cars at $t=2$.}
\label{fig:1B}
\end{center}
\end{figure}

%%%%%%%%%%%%%%%%%%%%%%%%%%%%%%%%%%%%%%%%
\subsection{Case 1C: $\rho^* < \rho_+ < \rho_- <1$.}
%%%%%%%%%%%%%%%%%%%%%%%%%%%%%%%%

Since  $\rho_->\rho^*$ is an unstable asymptote as $x\to-\infty$,
one must have 
\[
Q(x) \equiv \rho_- \qquad \mbox{for} \quad x < 0.
\]
Now consider the value $Q(0+)$. 
Since  $Q'(-\ell/\rho_-)=0$, equation~\eqref{DQ-2}  implies 
\[
V_- \phi(Q(-\ell/\rho_-)) = V_+ \phi(Q(0+)) \qquad \rightarrow \qquad 
Q(0+) < Q(-\ell/\rho_-)=Q(0-).
\]
This implies that $Q(x)$ is discontinuous at $x=0$, which is not possible 
for the solution of~\eqref{DQ-2}. 
We have the following Theorem.

\begin{theorem}\label{tm:1C}
Let $V_->V_+$ and $\rho^*< \rho_+<\rho_-<1$ 
with $f_-(\rho_-)=f_+(\rho_+)$.
There exists no profile $Q(x)$ that satisfies~\eqref{DQ-2} 
and the boundary conditions~\eqref{eq:ll}.
\end{theorem}

We remark that for this sub-case there exists a unique viscous profile for this case, 
see Figure~\ref{fig:1C} plot (2). 
We also plot the solution of the FtL model with this ``Riemann data'',
see Figure~\ref{fig:1C} plot (3). 
Observe that the solution is highly oscillatory on $x<0$, 
and it never settles, indicating no convergence as $t$ grows.

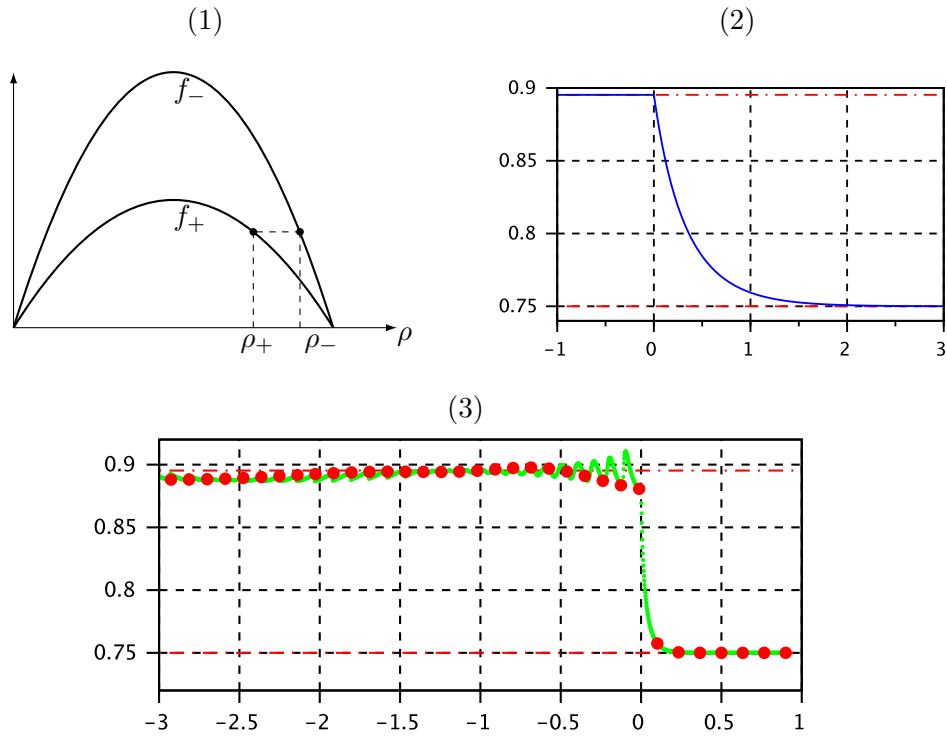
\begin{figure}[htbp]
\setlength{\unitlength}{0.85mm}
\begin{center}
\begin{tabular}{cc}
(1) & (2) \\
\begin{picture}(80,50)(-10,-5)  
\put(0,0){\vector(1,0){60}}\put(60,-2){$\rho$}
\put(0,0){\vector(0,1){40}}
\multiput(38,15)(2,0){4}{\line(1,0){1}}
\multiput(44.8,15)(0,-2){8}{\line(0,-1){1}}\put(44,-3){ $\rho_-$}
\multiput(37.5,15)(0,-2){8}{\line(0,-1){1}}\put(34,-3){ $\rho_+$}
\put(25,16){$f_+$}\put(25,36){$f_-$}
\put(44.8,15){\circle*{1.2}}\put(37.5,15){\circle*{1.2}}
\thicklines
\qbezier(0,0)(25,40)(50,0)
\qbezier(0,0)(25,80)(50,0)
\end{picture} &
\includegraphics[width=6.5cm,clip,trim=0mm 0mm 5mm 5mm]{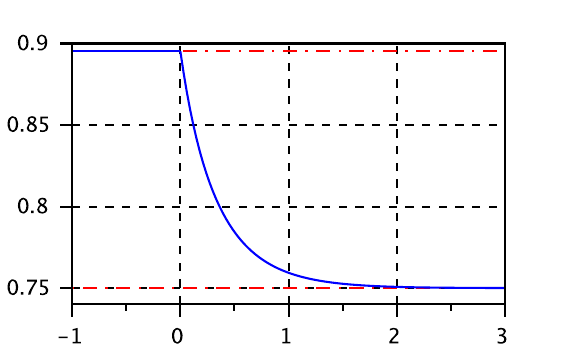}
\end{tabular} \\[3mm]
(3) \\
\includegraphics[width=11cm,clip,trim=0mm 0mm 5mm 5mm]{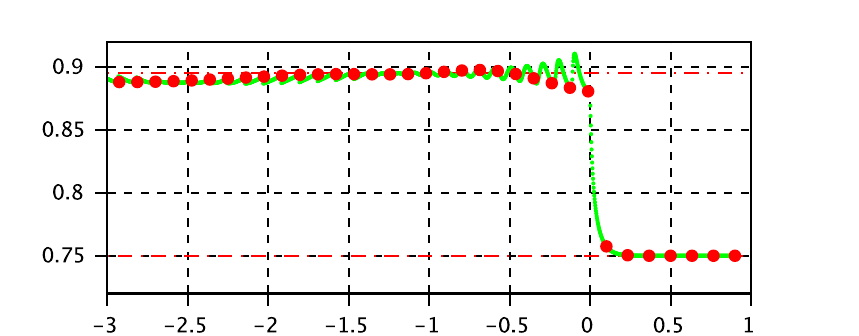}
\caption{Case 1C.  
(1) Plots of the flux functions and the locations of $\rho_-,\rho_+$; 
%(2) Plot of the unique stationary profile $Q(x)$ with $Q(0)=\rho_+$;
(2) Plot of the unique viscous profile $\rho^\ve(x)$ with $\rho^\ve(0)=\rho_-$; (3) Plot of the solution of the FtL model $\{z_i(t), \rho_i(t)\}$ with  a 
Riemann initial data. Here the thick dots are the locations
of cars at $t=2$. }
%(2): Viscous profile $\rho^\ve(x)$; (3): solution of the FtL model.}
\label{fig:1C}
\end{center}
\end{figure}

\bigskip

%%%%%%
\subsection{Case 1D:  $0< \rho_+ \le  \rho^* < \rho_-<1$.}

Since both $\rho_->\rho^*$ and $\rho_+ \le\rho^*$ are unstable asymptotes, 
one must have $Q(x)=\rho_-$ on $x<0$ and $Q(x)=\rho_+$ on $x>0$,
which is not possible. 

\begin{theorem}\label{tm:1D}
Let $V_->V_+$ and $0<\rho^+< \rho^*<\rho_-<1$ with $f_-(\rho_-)=f_+(\rho_+)$.
There exists no profile $Q(x)$ that satisfies~\eqref{DQ-2} 
and the boundary conditions~\eqref{eq:ll}.
\end{theorem}

For this sub-case there are no monotone viscous profiles either. 
In Figure~\ref{fig:1D} we plot numerical simulation result for the FtL model,
with ``Riemann initial data''. 
We observe oscillatory behavior on $x<0$, and a rarefaction wave
behavior on $x>0$.  The solution does not settle into any profile as $t$ grows.

\begin{figure}[htb]
\begin{center}
\setlength{\unitlength}{0.85mm}
\begin{picture}(57,50)(0,-3)  
\put(0,0){\vector(1,0){55}}
\put(0,0){\vector(0,1){40}}
\multiput(0,15)(2,0){26}{\line(1,0){1}}
\multiput(44.7,15)(0,-2){8}{\line(0,-1){1}}\put(44,-3){ $\rho_-$}
\multiput(12.4,15)(0,-2){8}{\line(0,-1){1}}\put(8,-3){ $\rho_+$}
\put(23,22){$f_+$}\put(23,36){$f_-$}
\put(44.8,15){\circle*{1.2}}\put(12.4,15){\circle*{1.2}}
\thicklines
\qbezier(0,0)(25,40)(50,0)
\qbezier(0,0)(25,80)(50,0)
\end{picture}
\includegraphics[width=9cm,clip,trim=5mm 0mm 21mm 6mm]{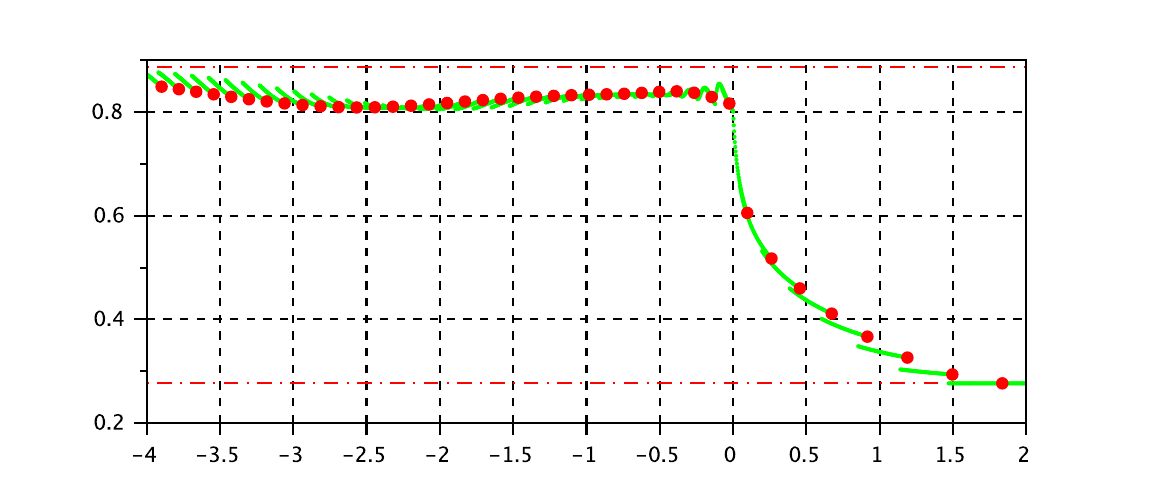}
\caption{Case 1D.  Left:  Plots of the flux functions and the locations of $\rho_-,\rho_+$; 
Right: Plot of the solution of the FtL model $\{z_i(t), \rho_i(t)\}$ with  a 
Riemann initial data. Here the thick dots are the locations
of cars at $t=2$. }
\label{fig:1D}
\end{center}
\end{figure}
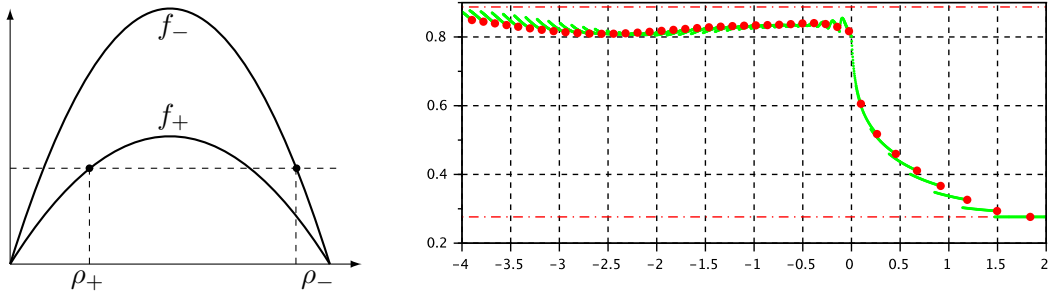

%%%%%%%%%%%%%%%%%%%%%%%%%%%%%%%%%%%%
\section{Case 2:  $V_- < V_+$.}\setcounter{equation}{0}%%
%%%%%%%%%%%%%%%%%%%%%%%%%%%%%%%%%%%%%

In this section we study the case where the speed limit has an upward jump at $x=0$.
The discussion for this case follows 
a similar path as for Case 1, but with rather different details. 
Given $\bar f$, which is in the range of both $f_\pm$, 
the candidates for $\rho_\pm$ are illustrated in Figure~\ref{fig:rhos2},
with
\[
 0<\rho^+_1 <\rho^-_1 \le \rho^* \le \rho_2^-  < \rho_2^+.
 \]

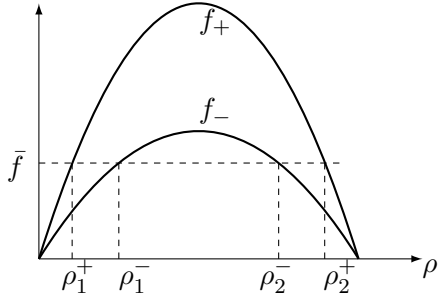
\begin{figure}[htbp]
\begin{center}
\setlength{\unitlength}{0.85mm}
\begin{picture}(70,50)(0,-5)  
\put(0,0){\vector(1,0){60}}\put(60,-2){$\rho$}
\put(0,0){\vector(0,1){40}}
\multiput(0,15)(2,0){24}{\line(1,0){1}}\put(-5,13){$\bar f$}
\multiput(5.2,15)(0,-2){8}{\line(0,-1){1}}\put(2,-4){ $\rho^+_1$}
\multiput(12.5,15)(0,-2){8}{\line(0,-1){1}}\put(11,-4){ $\rho_1^-$}
\multiput(37.5,15)(0,-2){8}{\line(0,-1){1}}\put(33,-4){ $\rho^-_2$}
\multiput(44.7,15)(0,-2){8}{\line(0,-1){1}}\put(43,-4){ $\rho^+_2$}
\put(25,22){$f_-$}\put(25,36){$f_+$}
\thicklines
\qbezier(0,0)(25,40)(50,0)
\qbezier(0,0)(25,80)(50,0)
\end{picture}
\caption{Flux functions $f_-, f_+$, and the locations of 
$\rho^-_1,\rho^+_1,\rho^-_2,\rho^+_2$.}
\label{fig:rhos2}
\end{center}
\end{figure}

We have the following 4  sub-cases:
\begin{itemize}
\item Case 2A:  
$\rho_-=\rho_1^-$ and $\rho_+=\rho^+_2$, 
such that $0<\rho_- < \rho^* < \rho_+<1$; 
\item Case 2B:  
$\rho_-=\rho_1^-$ and $\rho_+=\rho^+_1$, 
such that $ 0<\rho_+<\rho_-<\rho^*$;
\item Case 2C: 
$\rho_-=\rho_2^-$ and $\rho_+=\rho^+_2$, 
such that $ \rho^* \le \rho_-<\rho_+<1$;
\item Case 2D: 
$\rho_-=\rho_2^-$ and $\rho_+=\rho^+_1$, such that 
$ 0<\rho_+ < \rho^* \le \rho_-<1$.
\end{itemize}

%%%%%%%
\subsection{Case 2A:  $0 <\rho_- \le\rho^*< \rho_+<1$.}

Here both $\rho_-< \rho^* $ and $\rho_+>\rho^*$ are stable asymptotic limits 
as $x \to -\infty$ and $x\to +\infty$, respectively. 
Then, on $x>0$, the profile $Q(x)$  must be some horizontal shift of $W(x)$.
Using some horizontal shift of $W(x)$ 
as ``initial condition'', one can solve~\eqref{DQ} backward in $x$ on $x<0$. 
In next Theorem we establish unique solution of the initial value 
problem for~\eqref{DQ-2}, which in turn 
gives us the infinitely many profiles $Q(x)$ that satisfy the proper 
boundary conditions at the limit $x\to\pm\infty$. 

\begin{theorem}
Let $V_-<V_+$. Given $\rho_+$ such that  $\rho^*<\rho_+<1$.
Consider the initial value problem of ~\eqref{DQ-2}
%Let $Q(x)$ be the solution of~\eqref{DQ-2} on $x<0$,
with initial data given on $x\ge 0$ as some horizontal shift of 
$W(x)$,  with
%where $W(0)$ satisfies
\bel{Qc} 
\lim_{x\to\infty}W(x)=\rho_+, \qquad \rho_1^+ \le  W(0) \le \rho_2^-.
\eeq
Then the initial value problem has a unique solution $Q(x)$ on $x<0$.
%for this initial value problem.

Furthermore, such a solution satisfies also
\bel{LimL}
\lim_{x\to-\infty} Q(x) = \rho_-, \quad  \mbox{where} 
\quad \rho_- < \rho^*, \quad
f_-(\rho_-)=f_+(\rho_+).
\eeq
Piecing together $Q(x)$ on $x<0$ and $Q(x)=W(x)$ on $x\ge 0$, 
we obtain 
%such that each of these solutions is 
a solution to~\eqref{DQ-2}
with boundary conditions
\bel{bc2}
\lim_{x\to\infty}Q(x) = \rho_+,\qquad \lim_{x\to-\infty}Q(x) = \rho_-.
\eeq

Varying the $Q(0)$ value, always satisfying $\rho_1^+ \le Q(0) \le \rho_2^-$, 
one obtains infinitely many stationary wave profiles with the 
boundary conditions~\eqref{bc2}.
\end{theorem}

%%%%%%%%%%%%%%%
\begin{proof}
%%%%%%%%%%%%%%%
This Theorem is the counter part of Theorem~\ref{tm:IC}
and Corollary~\ref{cor:1} for Case 1A, 
but the proof here is much more involving due to the lack of monotonicity. 
See Figure~\ref{fig:2A}. 

%We first justify the assumption~\eqref{Qc}, in particular the upper bound 
%$Q(0) \le \rho_2^-$.  

Let the initial data be given on $x\ge0$ as some horizontal shift of $W(x)$
such that~\eqref{Qc} holds. 
Denote by $Q(x)$ the solution for this initial value problem, 
solved backward in $x$ for $x<0$.
Then $Q(x)$ is monotone  on $x\ge 0$ with $Q'(x) >  0$. 
Let $\{z_i\}$ be a car position distribution generated by $Q(x)$ 
with $z_0=0$ and 
\[
 z_k+\frac{\ell}{Q(z_k)} = z_{k+1}, \qquad \forall k\in\mathbb{Z}.
 \]
We also denote the intervals 
\[
I_k\;\dot=\; (z_{k}, z_{k+1}), \qquad  \mbox{for} ~~ k\in\mathbb{Z}.
\]
Throughout the rest of the proof, we use the simplified notations, for any index $k$, 
\bel{nota} 
Q_k=Q(z_k), 
\qquad 
\phi_k=\phi(Q(z_k)). 
\eeq

The proof takes several steps. 
\bigskip

\textbf{Step 1.} %%%%%%%%%%%
Assume that 
$Q(x)$ is a solution of the initial value problem, with the additional condition
\bel{Qcn}
\rho_- \le Q_0 \le \rho_2^-. 
\eeq
We claim that  
\bel{c-1}
Q'(0-) >0.
\eeq
Indeed, 
since $Q'(x)>0$ for $x>0$, by~\eqref{ss5} we have 
\bel{s7a}
\frac{1}{\phi_1}-\frac{1}{\phi_0} > Q_0 V_+ \left[ \frac{1}{\bar f} 
-\frac{1}{f_+(Q_0)}\right].
\eeq
By using $\bar f \le V_- Q_0 \phi_0$ and~\eqref{s7a}, we get 
\begin{eqnarray}
 \frac{1}{V_+\phi_1} 
-\frac{1}{V_- \phi_0}
 &=&  \frac{1}{V_+} 
\left[\frac{1}{\phi_1}  - \frac{1}{\phi_0}\right] + 
\frac{1}{V_+\phi_0} -\frac{1}{V_- \phi_0} \nonumber\\
& > &
Q_0 \left[\frac{1}{\bar f}  - \frac{1}{f_+(Q_0)}\right]+ 
\frac{1}{V_+\phi_0} -\frac{1}{V_- \phi_0} 
~\ge~ 0.\label{s7c}
\end{eqnarray}
Equation~\eqref{DQ-2} leads to
\[
 Q'(0-) = \frac{Q_0^2 V_+ \phi_1}{\ell } \left[ \frac{1}{V_+\phi_1} 
-\frac{1}{V_- \phi_0}\right] > 0,
\]
proving~\eqref{c-1}.

\textbf{Step 2.} %%%%%%%%%%%%%%%
We claim that on the interval $I_{-1}$ there doesn't exist any local maximum. 
Indeed, assume local maxima exist on $I_{-1}$, and let $y_1$ be the right most
local maximum, with $Q'(y_1)=0$. 
Let $y_1^\sharp >0$ be its leader. 
By~\eqref{DQ-2} and $Q'(y_1)=0$, we get
\bel{ss}
  V_- \phi(Q(y_1)) = V_+ \phi\left(Q(y_1^\sharp)\right). 
\eeq
Moreover,  there exists a point $y_2$, such that 
\[
 y_1<y_2<0, \qquad Q(y_2) < Q(y_1), \qquad Q'(y_2) <0.
 \]
Let $y_2^\sharp >0$ be its leader, where
$ y_2^\sharp > y_1^\sharp >0$.
Since $Q'(x)>0$ on $x>0$, we must have
\bel{ss2}
Q(y_2^\sharp) > Q(y_1^\sharp)  \qquad \Rightarrow \qquad
\phi\left(Q(y_2^\sharp)\right) <\phi\left( Q(y_1^\sharp)\right).
\eeq
On the other hand, by~\eqref{DQ-2} and $Q'(y_2)<0$, we get
\[
V_+ \phi\left(Q(y_2^\sharp)\right) > 
%V_- \phi_2  > V_- \phi_1 
V_- \phi(Q(y_2)) > V_- \phi(Q(y_1)) 
= V_+ \phi\left(Q(y_1^\sharp)\right), 
\]
a contradiction to~\eqref{ss2}.

\bigskip

\textbf{Step 3.} %%%%%%%%%
We now show that, if~\eqref{Qcn} holds, then
\bel{s8a}
  Q_{-1} < Q_0.
\eeq 
Indeed, we know that there are no local maxima on $I_{-1}$ and 
$Q'(0-)>0$. If $Q(x)$ is monotone increasing on $I_{-1}$,
%then an induction argument shows that $Q(x)$ is monotone for $x<0$, 
then~\eqref{s8a}  trivially holds.
Now consider the case where $Q(x)$  has a local minimum.
We prove by contradiction. Assume  
that there exist a point $y \in(z_{-1}, 0)$ where
\[
 Q(y) = Q(0)=Q_0, \qquad 
Q(x) < Q_0 \quad \mbox{for } x\in(y,0). 
\]
Let $y^\sharp$ be its leader, where $0<y^\sharp<z_1$. 
Recall~\eqref{ss-a}, we have
\begin{eqnarray*}
&&\hspace{-1cm} \int_y^{y^\sharp} 
\left[ \frac{1}{k(z) \phi(Q(z))} -\frac{1}{V_- \phi(Q(y))} \right]dz
~= ~
\int_y^{y^\sharp} 
\left[ \frac{1}{k(z) \phi(Q(z))} -\frac{1}{V_- \phi_0} \right]dz\\
&=& \frac{\ell}{\bar f} - \frac{\ell}{f_-(Q(y))} 
~=~ \frac{\ell}{\bar f} - \frac{\ell}{f_-(Q_0)}  \;\dot=\; \gamma ~\ge~0,
\end{eqnarray*}
which gives
\[
\gamma= \int_y^0 \left[\frac{1}{V_-\phi(Q(z))} -\frac{1}{V_- \phi_0} \right] dz
+ \int_0^{y^\sharp} \left[\frac{1}{V_+\phi(Q(z))} -\frac{1}{V_- \phi_0}  \right]dz.
\]

Since the first integrand on the right hand side is strictly negative, we get
\bel{s8aa}
\int_0^{y^\sharp} \left[\frac{1}{V_+\phi(Q(z))} -\frac{1}{V_- \phi_0}  \right]dz>\gamma.
\eeq
But \eqref{s8aa} is not possible. 
Indeed, since $Q'(x)>0$ on $x>0$, the mapping 
$x\mapsto (1/\phi(Q(x)))$ is increasing. 
Using  that 
\[
\frac{1}{V_+ \phi(Q_0)} - \frac{1}{V_- \phi(Q_0)} <0,
\qquad 
\int_0^{z_1}  \left[\frac{1}{V_+ \phi(Q(z))} - \frac{1}{V_- \phi(Q_0)} \right] dz
= \gamma,
\]
one reaches
\[
\int_0^{x}  \left[\frac{1}{V_+ \phi(Q(z))} - \frac{1}{V_- \phi(Q_0)} \right] dz
< \gamma,\qquad \mbox{for any } ~ x\in(0,z_1), 
\]
 a contradiction to~\eqref{s8aa}.

\bigskip

\textbf{Step 4.} %%%%%%%%% 
We now have that, for the initial value problem with initial data $W(x)$ 
on $x\ge 0$ satisfying~\eqref{Qcn}, the solution $Q(x)$, 
defined on $x<0$,  satisfies 
\bel{s4aa}
0 < Q(z_{-1}) < \rho_2^-.
\eeq
We now claim that there exists a unique solution $Q(x)$ 
for the initial value problem, which satisfies
\bel{s4bb}
\lim_{x\to-\infty} Q(x) = \rho_-.
\eeq

Indeed, if $Q(x)$ stays on one side of $\rho_-$ 
on an interval $I_k$ for some index $k\le -2$, 
then Lemma~\ref{lm5} provides the results. 
Now consider the case the $Q(x)$ is oscillatory and crosses $\rho_-$ 
at least once on each interval $I_k$, for $k\le -2$. 
We apply a similar argument as the proof for Lemma~\ref{lm5}. 
Let 
\[
M_k = \max\left\{ \max_{x\in I_k} \frac{1}{\phi(Q(x))},  \frac{1}{\phi(\rho_-)} \right\} .
\]
Then, we have, for some index $k\le 2$,  
\[
M_k = \frac{1}{\phi(Q(y_k))} > \frac{1}{\phi(\rho_-)}, \qquad \mbox{where}~~
y_k\in I_k \quad \mbox{and}~~Q'(y_k)=0. 
\]
%where $y_k$ is a local maximum with $Q'(y_k)=0$. 
%$M_k > 1/\phi(\rho_-)$ for some $k$, and let $y_k\in I_k$ be the point where
%$ M_k = \phi(Q(y_k))^{-1}$, and 
Let $y_k^\sharp=y_k + \ell/Q(y_k)$ denote the position of
the leader for the car at $y_k$. 
By Lemma~\ref{lm4.5} we have $y_k^\sharp\in I_{k+1}$. 
Then $Q'(y_k)=0$ implies that $Q(y_k)=Q(y_k^\sharp)$, 
and~\eqref{ss5} implies 
\[
M_{k+1}-M_k  \ge V_- Q(y_k) 
\left[ 
\frac{1}{f_-(\rho_-)} - \frac{1}{f_-(Q(y_k))}
%\frac{1}{\rho_- \phi(\rho_-)} - \frac{1}{Q(y_k)\phi(Q(y_k))}
\right]
= \mathcal{O}(1) \cdot (Q(y_k)-\rho_-).
%\left[\frac{1}{\phi(Q(y_k))}-\frac{1}{\phi(\rho_-)}\right].
\]
Thus, we conclude that 
\[ \lim_{k\to-\infty} Q(y_k)=\rho_-, \quad \mbox{and}\quad 
  \lim_{k\to-\infty} M_k = \frac{1}{\phi(\rho_-)}.\] 
Therefore on $x\le 0$ there exists an upper envelope $E^\sharp (x)$ for $Q(x)$, such that
\bel{eq:E}
 Q(x) \le E^\sharp(x), \qquad \lim_{x\to-\infty} E^\sharp(x)=\rho_-.
\eeq

A symmetrical argument for the local minima below $\rho_-$ leads to a lower envelope
$E^\flat(x)$ on $x<0$ for $Q(x)$, with 
\bel{eq:EE}
  E^\flat (x) < \rho_-, \qquad \lim_{x\to-\infty}E^\flat(x) =\rho_-.
\eeq
The result~\eqref{s4bb} follows from a squeezing argument.
Finally, the uniqueness of the solution follows from the 
transversality properties~\eqref{eq:cc1}-\eqref{eq:cc}, see~\cite{MR964856}. 

Piecing together the solution $Q(x)$ on $x<0$ with the initial data 
$Q(x)=W(x)$ on 
$x\ge 0$,  we obtain a stationary profile, calling it again by $Q(x)$ for 
$x\in\mathbb{R}$, that satisfies the DDDE~\eqref{DQ-2} and the boundary 
conditions~\eqref{bc2}. 
Thus, we obtain infinitely many profiles for $Q(x)$, one for
 each $Q(0)$ value satisfying~\eqref{Qcn}.

\bigskip

\textbf{Step 5.}
Denote by  $Q^\sharp(x)$ the unique profile with $Q^\sharp(0)=\rho_2^-$. 
By Step 3, we have 
\[
0< Q^\sharp(z_{-1}) < Q^\sharp(0) =\rho_2^-.
\]
%By Step 4, such a profile exists, is unique and satisfies~\eqref{s4bb}. 
%By the transversality property, the solution is unique.

We now relax the condition~\eqref{Qcn} on $Q(0)$ to~\eqref{Qc}, i.e, 
  $ \rho_1^+ < Q(0) < \rho_2^-$.  
Indeed, any profile $Q(x)$ with $ \rho_1^+ < Q(0) < \rho_2^-$ will lie below 
$Q^\sharp(x)$, with
\[
0<Q(z_{-1})< \rho_2^-.
\]
By Step 4, such a profile satisfies the boundary condition~\eqref{s4bb}, 
completing the proof.
%%%%%%%%%%%%%%%%%
\end{proof}
%%%%%%%%%%%%%%%%%

\begin{remark}
We remark on the bound~\eqref{Qc}, in particular the 
upper bound $Q_0 \le \rho_2^-$, which is different from
Case~1A in section~3.1.  
First, we show that the constant solution 
$Q(x)\equiv \rho_+$  on $x\ge 0$ is not valid. 
Indeed, with $Q_0=Q_1=\rho_+$, we have 
\[Q'(0-)  = \frac{\rho_+^2}{\ell V_- \phi(\rho_+)} (V_--V_+) \phi(\rho_+) <0.\]
Then, on the interval $I_{-1}=[z_{-1},z_0]$, 
$V_- \phi(Q(x)) < V_+\phi(\rho_+)$, so $Q'(x)<0$.  
By induction argument one concludes that $Q'(x)<0$ for $x<0$. 
In fact, numerical simulation shows that $Q(x)$ blows up to $\infty$
at finite $\bar x <0$ as $x$ goes backwards. 

With the upper bound $Q_0 \le \rho_2^-$ we have~\eqref{c-1}, 
and we ensure that $Q(x) < \rho_2^-$ on $x<0$, 
and consequently the asymptotic limit of $\rho_-$ as $x\to-\infty$. 
It is possible that this upper bound could be somewhat relaxed,
but a sharp bound is difficult to find. 
\end{remark}

Sample profiles of $Q(x)$ are plotted in Figure~\ref{fig:2A} plot (2),
where we observe that the profiles are not monotone.
We also plot multiple viscous profiles $\rho^\ve(x)$
in Figure~\ref{fig:2A} plot (3), as a comparison. 
Note that if $\rho^\ve(0) \in(\rho_-,\rho_+)$,
the viscous profiles are monotone, a property 
not preserved by $Q(x)$.

\begin{figure}[htb]
\begin{center}
\begin{tabular}{cc}
(1) & (2) \\
\setlength{\unitlength}{0.85mm}
\begin{picture}(84,50)(-12,-5)  
\put(0,0){\vector(1,0){60}}
\put(0,0){\vector(0,1){40}}
\multiput(5,15)(2,0){21}{\line(1,0){1}}
\put(12.5,15){\circle*{1.2}}\put(44.7,15){\circle*{1.2}}
\multiput(44.8,15)(0,-2){8}{\line(0,-1){1}}\put(42,-3){ $\rho_+$}
\multiput(12.5,15)(0,-2){8}{\line(0,-1){1}}\put(10,-3){ $\rho_-$}
\multiput(37.5,15)(0,-2){8}{\line(0,-1){1}}\put(35,-4){ $\rho_2^-$}
\multiput(5,15)(0,-2){8}{\line(0,-1){1}}\put(1,-4){ $\rho_1^+$}
\put(25,16){$f_-$}\put(25,36){$f_+$}
\thicklines
\qbezier(0,0)(25,40)(50,0)
\qbezier(0,0)(25,80)(50,0)
\end{picture}
&
\includegraphics[width=6.5cm,clip,trim=5mm 0mm 10mm 6mm]{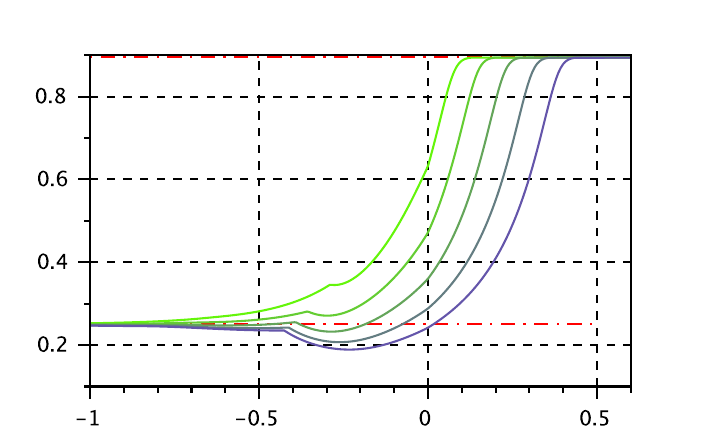}\\[3mm]
(3) & (4) \\
\includegraphics[width=6.5cm,clip,trim=5mm 0mm 10mm 6mm]{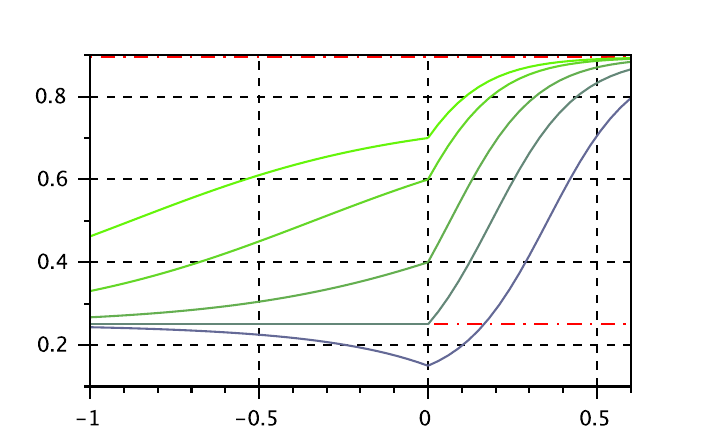}&
\includegraphics[width=6.5cm,clip,trim=5mm 0mm 10mm 6mm]{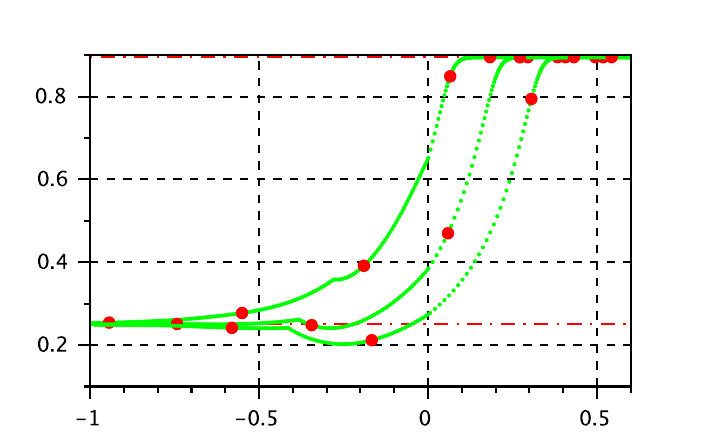}
\end{tabular}
\caption{Case 2A. (1) Flux functions and the locations of $\rho_-,\rho_+$; 
(2) Plots of various profiles of $Q(x)$, with different values of $Q(0)$; 
(3) Plots of various viscous traveling waves $\rho^\ve(x)$, with different values
of $\rho^\ve(0)$; 
(4) Plots of various solutions of the FtL model $\{z_i(t), \rho_i(t)\}$,  with 3 different initial Riemann data. 
Here the thick dots denote the locations of cars at $t=2$.}
\label{fig:2A}
\end{center}
\end{figure}
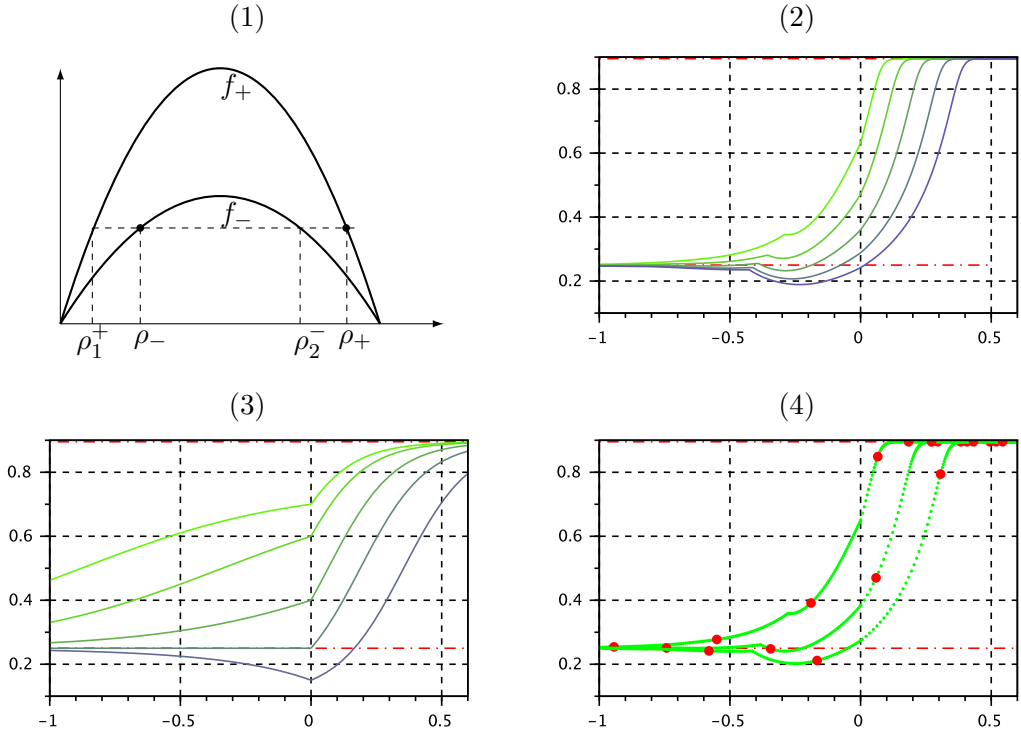

\medskip

\textbf{Local Stability of the Profiles.} 
%As in the setting of Theorem~\ref{tm:att},   
Let $Q^\sharp(x)$ be the profile with 
$Q^\sharp(0)=\rho_2^-$,
and let $Q^\flat(x)$ be the limit profile as $Q(0) \to \rho_1^+$.
Similar to Case 1A, we define a basin of attraction $D$
as~\eqref{def:D}. 
All profiles lie in $D$, and they do not intersect with 
each other. 
Parametrizing the region with these profiles, as in Theorem~\ref{tm:att},
we get the same local stability property. 
We skip the details. 

\medskip
Again,  numerical simulations are performed for the FtL model for Case 2A, 
and we plot the solutions with ``Riemann initial data''~\eqref{eq:rmzr}. 
See Figure~\ref{fig:2A} plot (4).
We see the clear convergence to a certain profile for each choice of 
initial data.

%%%%%%%%%%%%%%%%%%%%%%%%%%%%%%%%%%%%%%%%%%
\subsection{Case 2B:  $ 0<\rho_+<\rho_-<\rho^*$.}
%%%%%%%%%%%%%%%%%%%%%%%%%%%%%%%%%%%%%%%%%%

This is similar to Case 1B. 
Since $\rho_+$ is an unstable asymptote for $x\to\infty$, 
we must have $Q(x)\equiv \rho_+$ on $x\ge 0$.  
Using this as the initial data, one can solve $Q(x)$ backward in $x$.
Since $\rho_-$ is a stable asymptote, we have $Q(x)\to \rho_-$
as $x\to-\infty$. 
Thus there exists a unique monotone profile $Q(x)$.  
For the same reason as for case 1B, this profile 
is not a local attractor for the  solutions of the FtL model. 

In Figure~\ref{fig:2B} we plot the profile $Q(x)$ in plot (2), 
the viscous profile $\rho^\ve(x)$ in plot (3), 
and the solution of the FtL model with ``Riemann initial data'' in plot (4). 
Note that a perturbation enters the region $x>0$, even with initial Riemann
data, indicating the instability of the profile $Q(x)$. 

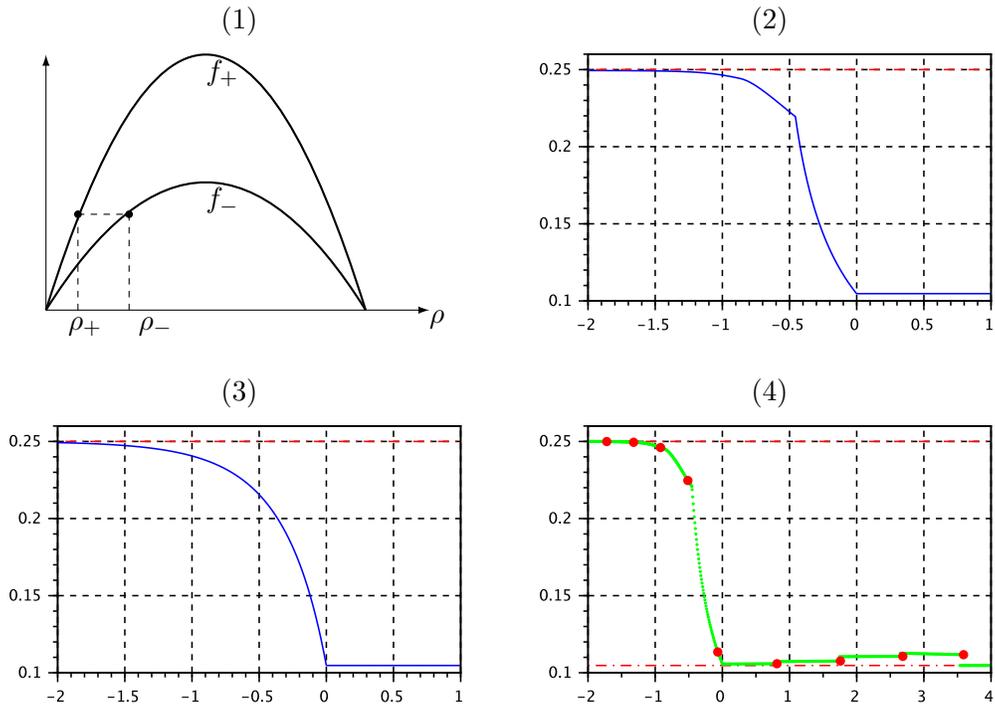
\begin{figure}[htbp]
\begin{center}
\begin{tabular}{cc}
(1) & (2) \\
\setlength{\unitlength}{0.85mm}
\begin{picture}(78,40)(-8,-5)  
\put(0,0){\vector(1,0){60}}\put(60,-2){$\rho$}
\put(0,0){\vector(0,1){40}}
\multiput(5,15)(2,0){4}{\line(1,0){1}}
\multiput(5,15)(0,-2){8}{\line(0,-1){1}}\put(2,-3){ $\rho_+$}
\multiput(13,15)(0,-2){8}{\line(0,-1){1}}\put(13,-3){ $\rho_-$}
\put(25,16){$f_-$}\put(25,36){$f_+$}
\put(5,15){\circle*{1.2}}\put(13,15){\circle*{1.2}}
\thicklines
\qbezier(0,0)(25,40)(50,0)
\qbezier(0,0)(25,80)(50,0)
\end{picture} &
\includegraphics[width=6.5cm,clip,trim=1mm 0mm 10mm 6mm]{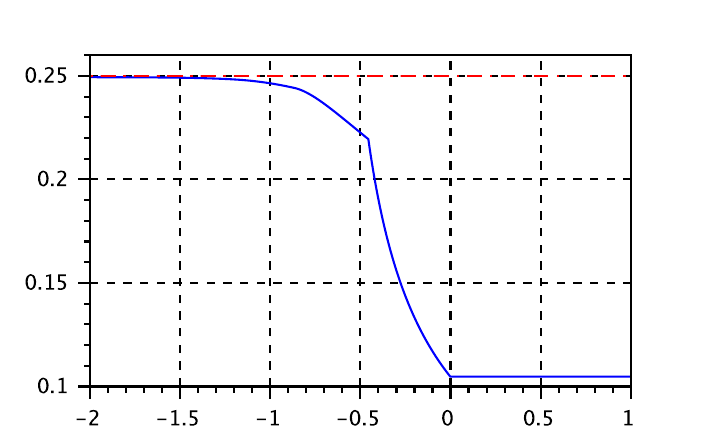}\\[3mm]
(3) & (4) \\
\includegraphics[width=6.5cm,clip,trim=1mm 0mm 10mm 6mm]{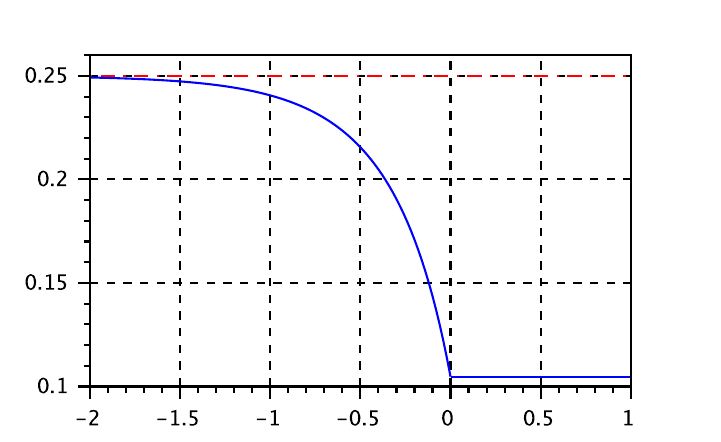} &
\includegraphics[width=6.5cm,clip,trim=1mm 0mm 10mm 6mm]{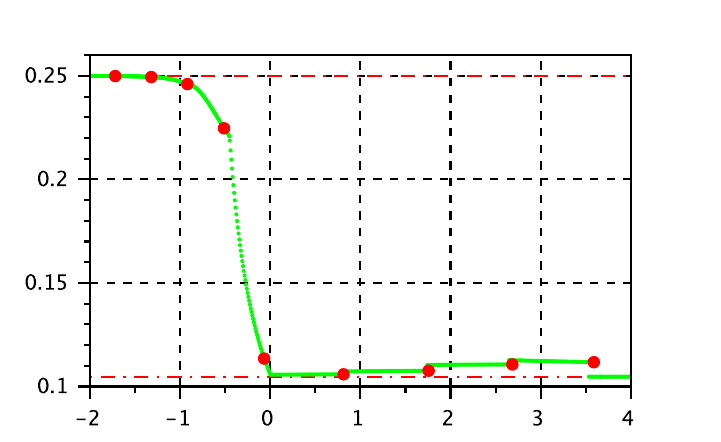}
\end{tabular}
\caption{Case 2B.
(1) Flux functions and the locations of $\rho_-,\rho_+$; 
(2) Plots of the unique profile of $Q(x)$, with  $Q(0)=\rho_+$; 
(3) Plots of various viscous traveling waves $\rho^\ve(x)$, with 
 $\rho^\ve(0)=\rho_+$; 
(4) Plots of the solution of the FtL model $\{z_i(t), \rho_i(t)\}$
  with a  Riemann initial data. 
Here the thick dots denote the locations of cars at $t=2$.
}
\label{fig:2B}
\end{center}
\end{figure}

%%%%%%%
\subsection{Case 2C: $\rho^* \le \rho_-  <  \rho_+<1$.}

This is the corresponding sub-case as for Case 1C. 
With the same argument, one concludes that there doesn't exist any profile 
$Q(x)$, although a viscous profile $\rho^\ve(x)$ does exist.  
See Figure~\ref{fig:2C}  plot (2). 
The solution of the FtL model in Figure~\ref{fig:2C} plot (3) 
demonstrates severe oscillation on $x<0$ which never settles as $t$ grows.

\begin{figure}[htbp]
\begin{center}
\begin{tabular}{cc}
(1) & (2) \\
\setlength{\unitlength}{0.85mm}
\begin{picture}(70,50)(0,-5)  
\put(0,0){\vector(1,0){60}}\put(60,-2){$\rho$}
\put(0,0){\vector(0,1){40}}
\multiput(38,15)(2,0){4}{\line(1,0){1}}
\multiput(44.8,15)(0,-2){8}{\line(0,-1){1}}\put(44,-3){ $\rho_+$}
\multiput(37.5,15)(0,-2){8}{\line(0,-1){1}}\put(34,-3){ $\rho_-$}
\put(25,16){$f_-$}\put(25,36){$f_+$}
\put(44.8,15){\circle*{1.2}}\put(37.5,15){\circle*{1.2}}
\thicklines
\qbezier(0,0)(25,40)(50,0)
\qbezier(0,0)(25,80)(50,0)
\end{picture} &
\includegraphics[width=6.5cm,clip,trim=0mm 0mm 8mm 5mm]{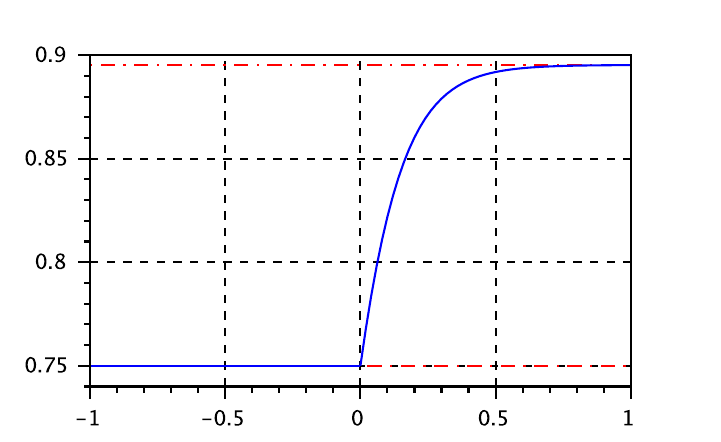}
\end{tabular}\\[3mm]
(3) \\
\includegraphics[width=6.5cm,clip,trim=0mm 0mm 8mm 5mm]{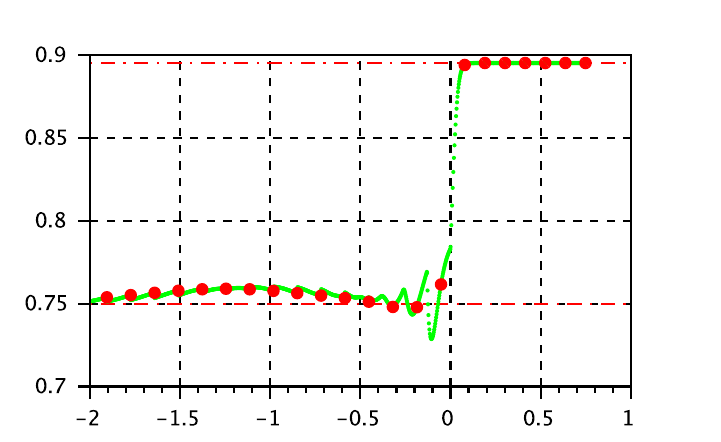}
\caption{Case 2C. 
(1) Plot of the flux functions $f_-,f_+$ 
and the locations of $\rho_-,\rho_+$; 
(2) Plot of the  unique viscous profile $\rho^\ve$ with 
$\rho^\ve(0)=\rho_-$; 
(3)   Plot of the solution of the FtL model $\{z_i(t), \rho_i(t)\}$ with  a 
Riemann initial data. Here the thick dots are the locations
of cars at $t=2$.} 
%a solution for the FtL model.}
\label{fig:2C}
\end{center}
\end{figure}
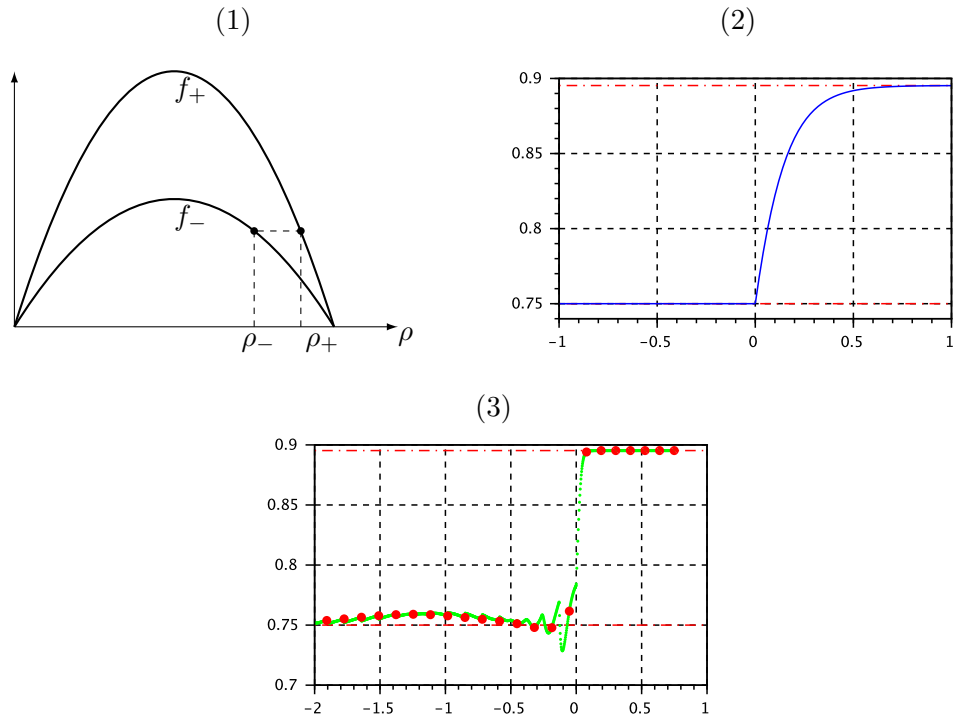

%%%%%%%%%%%%%%
\subsection{Case 2D:  $ 0<  \rho_+<\rho^*\le \rho_-<1$.}
%%%%%%%%%%%%%%
For this case, we have neither the profile $Q(x)$ nor the 
viscous profile $\rho^\ve(x)$. 
We plot a solution of the FtL model in Figure~\ref{fig:2D}, 
with ``Riemann initial data''.
We see that the solution of the FtL model
 doesn't converge to any limit as time grows. 

\begin{figure}[htbp]
\begin{center}
\setlength{\unitlength}{0.85mm}
\begin{picture}(70,50)(0,-5)  
\put(0,0){\vector(1,0){60}}\put(60,-2){$\rho$}
\put(0,0){\vector(0,1){40}}
\multiput(0,15)(2,0){24}{\line(1,0){1}}
\multiput(5.2,15)(0,-2){8}{\line(0,-1){1}}\put(3,-3){ $\rho_+$}
\multiput(37.5,15)(0,-2){8}{\line(0,-1){1}}\put(34,-3){ $\rho_-$}
\put(25,16){$f_-$}\put(25,36){$f_+$}
\put(5.2,15){\circle*{1.2}}\put(37.5,15){\circle*{1.2}}
\thicklines
\qbezier(0,0)(25,40)(50,0)
\qbezier(0,0)(25,80)(50,0)
\end{picture}
\includegraphics[width=7cm,clip,trim=0mm 0mm 8mm 5mm]{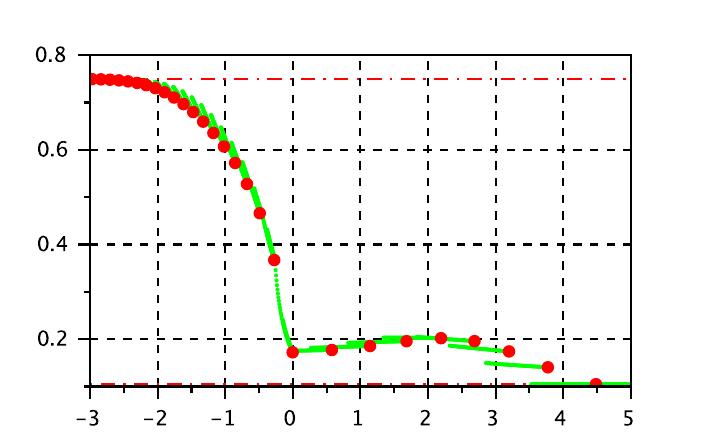}
\caption{Case 2D. Left: Plots of the flux functions and the locations of $\rho_-,\rho_+$; 
Right: Plot of the solution of the FtL model $\{z_i(t), \rho_i(t)\}$ with  a 
Riemann initial data. Here the thick dots are the locations
of cars at $t=2$.}
%Solution of the  FtL model.}
\label{fig:2D}
\end{center}
\end{figure}
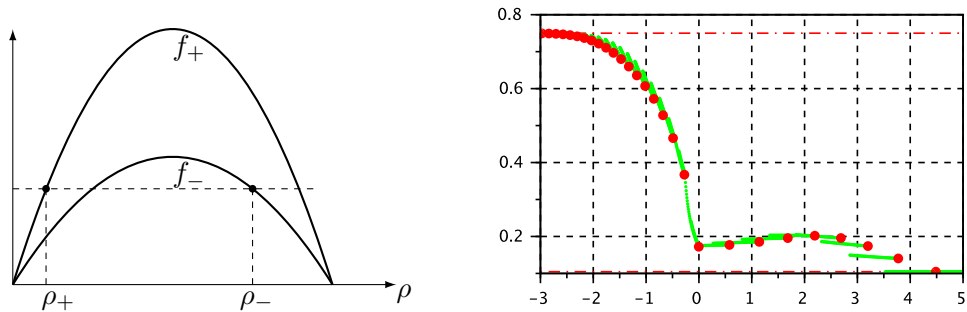

%%%%%%%%%%%%%%%%%%%%
\section{A Numerical Simulation}
%%%%%%%%%%%%%%%%%%%%%%

We perform  numerical simulation to obtain approximate solution
for the FtL model, with ``Riemann'' initial data $(\rho^L,\rho^R)$ such that
\[
\rho_i(0) = \begin{cases}
\rho^R,  \quad & i\ge 0, \\
\rho^L,  \quad & i< 0 ,
\end{cases}
\qquad
z_i(0) = \begin{cases}
i \ell/\rho^R,  \quad & i\ge 0, \\
i \ell/\rho^L,  \quad & i< 0 ,
\end{cases} 
\qquad z_0(0)=0. 
\]
We choose values of $(\rho^L,\rho^R)$ such that
\[ f_-(\rho^L) \not= f_+(\rho^R).\]
We use 
\[
 \phi(\rho)=1-\rho, \quad
(V_-, V_+)=(2,1), \quad 
\rho^L= 0.6, \quad \rho^R=0.7,\qquad \ell=0.01.
\]
%The solution at $t=1$ is shown in Figure~\ref{fig:Ex1} plot (2). 

The flux functions $f_-,f_+$ and the locations of $\rho^{L,R}$ 
are illustrated in Figure~\ref{fig:Ex1} plot (1),
while the solution $\{z_i(T),\rho_i(T),\}$ of the FtL model 
is shown in plot (2).
As a comparison, we also simulate the viscous conservation law
\[
 \rho_t + f(k(x),\rho)_x = \ve \rho_{xx}, 
 \]
using the same Riemann data, with $\ve= 0.02$ and $k(x)$
the jump function~\eqref{eq:V0}. 
The result is shown in plot (3).

\begin{figure}[htb]
\begin{center}
\begin{tabular}{cc}
(1)&(2)\\
\includegraphics[width=6.5cm,clip,trim=0mm 6.8cm 10mm 77mm]{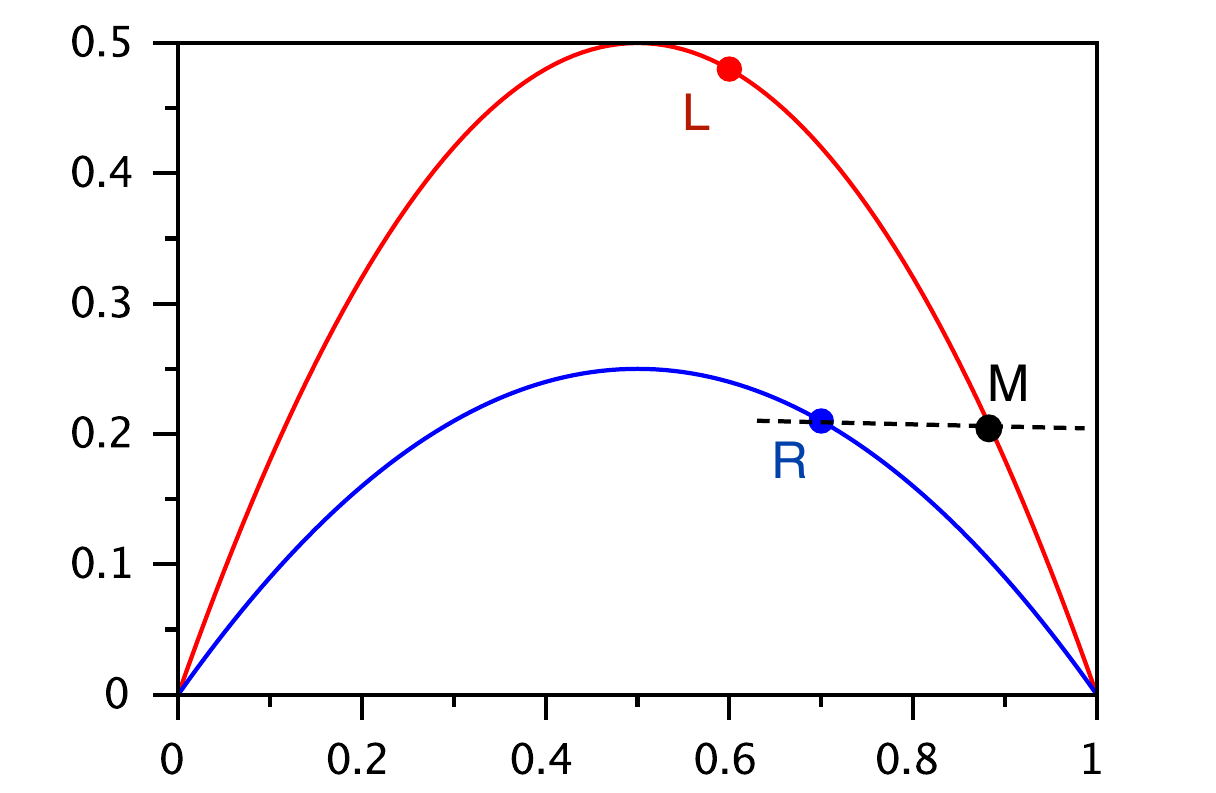}
&
\includegraphics[width=6.5cm,clip,trim=0mm 0cm 0mm 6mm]{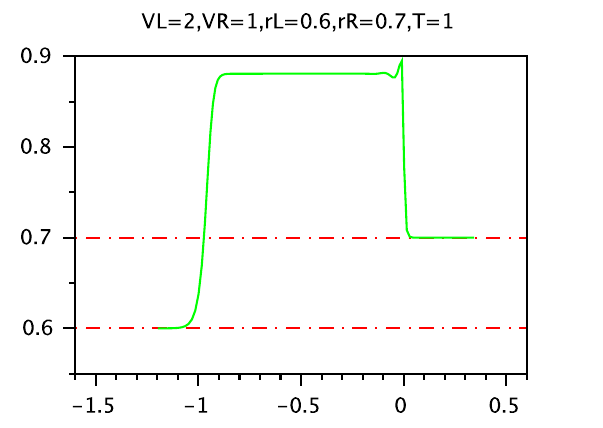}
\end{tabular} \\[2mm]
(3) \\
\includegraphics[width=6.5cm,clip,trim=0mm 0cm 0mm 6mm]{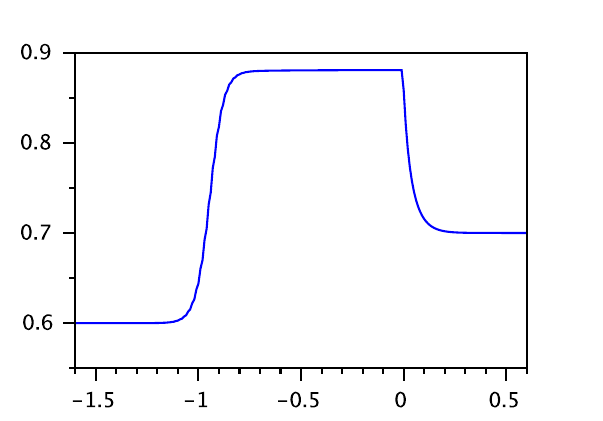}
\caption{(1). Plots of the flux functions and the location of the left (L), right (R) and middle (M) states in the solution of the Riemann problem; 
(2).  Numerical simulation results $\{z_i(t), \rho_i(t)\}$ 
with FtL model with Riemann initial data, at $t=1$; 
(3) Numerical simulation results $\rho^\ve(t)$ for  
the viscous conservation law at $t=1$, with the same Riemann initial data.}
\label{fig:Ex1}
\end{center}
\end{figure}

The vanishing viscosity limit solution for the conservation law~\eqref{claw} 
consists of a shock with negative speed from L to M, and a stationary
jump from M to R. 
The solution of the FtL model captures this main feature. 
However, due to the instability of the path M-R (where the left state is unstable), 
we observe oscillations behind the stationary jump at $x=0$. 
We remark that 
the solution of the viscous conservation law with the same initial data
does not contain oscillation behind $x=0$.

%%%%%%%%%%%%%%%%%%%%
\section{Concluding Remarks}

In this paper we derive a discontinuous delay differential equation 
for the stationary traveling wave profile for 
an ODE model of traffic flow, 
where the road condition is discontinuous. 
For various cases, we obtain results on the existence, uniqueness and 
local stability of the profiles. 

These results offer alternative approximate solutions to the 
scalar conservation law with discontinuous flux,
as a counter part to the classical vanishing viscosity approach. 
The stabilizing effect of the viscosity is not entirely present in the 
FtL model, where oscillations are observed behind the discontinuity 
in the road condition.
This is caused by the ``directional'' influence in real life traffic,
where the drivers adjust their behavior only according to situations ahead of them,
not what is behind.
Heuristically, this fact contributes to the ``lack of viscosity'' behind the jump at $x=0$,
and thus the oscillations. 

The natural followup work is to investigate the convergence of solutions of the FtL model,
under suitable assumptions, to some entropy admissible solution of the scalar conservation law with
discontinuous flux.  We expect this to be a challenging task, 
due to the non-monotone profiles and oscillations behind the jump in the road condition. 

One may criticize the FtL model used here of being too simple,
especially around the jump in the road condition, 
where the drivers change their speeds suddenly as they cross $x=0$. 
The model is a first order approximation where one assumes instant acceleration.
A high order model, where the acceleration is finite, might smooth out
the behavior near $x=0$ and remove the oscillations.
However, such model would take the velocities of the cars as unknowns, 
and thus become much more complex. 

\bigskip

%\textbf{Acknowledgement.} 
%The author is grateful to an anonymous referee for
%careful reading of the first manuscript and detailed comments,
%which led to the improvement of the manuscript.

\bibsection

\end{document}